\documentclass[12pt,a4paper]{article}

\usepackage[margin=1in]{geometry}

\usepackage[english,activeacute]{babel}
\usepackage{etex}
\usepackage{latexsym}
\usepackage{mathtools}
\usepackage{amsmath}
\usepackage{amssymb}
\usepackage{wasysym}
\usepackage{mathrsfs}
\usepackage{euscript}
\usepackage{amsthm}
\usepackage{amsfonts}
\usepackage{cancel}
\usepackage{multicol}
\usepackage{color}
\usepackage{multirow}
\usepackage{makeidx}
\usepackage{upgreek}
\usepackage{setspace} 
\usepackage{import}
\usepackage{float}
\usepackage{mleftright}
\usepackage{stmaryrd}
\usepackage{diagbox}

\usepackage{bbold}

\usepackage{tikz}
\usepackage{tikz-cd}
\usetikzlibrary{matrix,arrows,calc,decorations.pathmorphing,decorations.pathreplacing,positioning}

\usepackage{xcolor}
\definecolor{babyblueeyes}{rgb}{0.54, 0.81, 0.94}
\definecolor{blizzardblue}{rgb}{0.67, 0.9, 0.93}
\definecolor{blue(munsell)}{rgb}{0.0, 0.5, 0.69}
\definecolor{bluegray}{rgb}{0.4, 0.6, 0.8}
\definecolor{bondiblue}{rgb}{0.0, 0.58, 0.71}
\definecolor{cadetblue}{rgb}{0.37, 0.62, 0.63}
\definecolor{carolinablue}{rgb}{0.6, 0.73, 0.89}
\definecolor{cinnamon}{rgb}{0.82, 0.41, 0.12}
\definecolor{darkcandyapplered}{rgb}{0.64, 0.0, 0.0}
\definecolor{darkcyan}{rgb}{0.0, 0.55, 0.55}
\definecolor{darkmidnightblue}{rgb}{0.0, 0.2, 0.4}
\definecolor{darkpastelblue}{rgb}{0.47, 0.62, 0.8}
\definecolor{frenchblue}{rgb}{0.0, 0.45, 0.73}

\usepackage{url}
\usepackage[bookmarksopen,colorlinks,allcolors=frenchblue,urlcolor=frenchblue]{hyperref}

\makeindex

\usepackage{scalerel}

\usepackage{wrapfig}
\usepackage{ifpdf}

\usepackage{syntonly}

\DeclareMathOperator{\id}{\mathsf{id}}

\DeclareMathOperator{\hocolim}{\mathsf{hocolim}}
\DeclareMathOperator{\holim}{\mathsf{holim}}

\DeclareMathOperator{\source}{\mathsf{s}}
\DeclareMathOperator{\target}{\mathsf{t}}

\DeclareMathOperator{\upf}{\textit{f}}
\DeclareMathOperator{\upg}{\textit{g}}
\DeclareMathOperator{\uph}{\textit{h}}
\DeclareMathOperator{\upi}{\textit{i}}
\DeclareMathOperator{\upj}{\textit{j}}
\DeclareMathOperator{\upn}{\textup{n}}

\DeclareMathOperator{\upp}{\textit{p}}
\DeclareMathOperator{\upq}{\textit{q}}

\DeclareMathOperator{\D}{\EuScript{D}}
\DeclareMathOperator{\Qstab}{\EuScript{Q}}

\DeclareMathOperator{\M}{\EuScript{M}}
\DeclareMathOperator{\OpN}{\EuScript{N}}

\DeclareMathOperator{\V}{\EuScript{V}}

\DeclareMathOperator{\Prect}{\mathsf{P}}
\DeclareMathOperator{\Frect}{\mathsf{F}}
\DeclareMathOperator{\Srect}{\mathsf{S}}

\DeclareMathOperator{\Xrect}{\mathsf{X}}
\DeclareMathOperator{\Yrect}{\mathsf{Y}}
\DeclareMathOperator{\Zrect}{\mathsf{Z}}
\DeclareMathOperator{\Trect}{\mathsf{T}}
\DeclareMathOperator{\Nrect}{\mathsf{N}}
\DeclareMathOperator{\Mrect}{\mathsf{M}}
\DeclareMathOperator{\Qrect}{\mathsf{Q}}

\DeclareMathOperator{\Arect}{\mathsf{A}}
\DeclareMathOperator{\Brect}{\mathsf{B}}
\DeclareMathOperator{\Crect}{\mathsf{C}}
\DeclareMathOperator{\Vrect}{\mathsf{V}}
\DeclareMathOperator{\Wrect}{\mathsf{W}}
\DeclareMathOperator{\Urect}{\mathsf{U}}
\DeclareMathOperator{\Drect}{\mathsf{D}}

\DeclareMathOperator{\Irect}{\mathsf{I}}
\DeclareMathOperator{\Jrect}{\mathsf{J}}
\DeclareMathOperator{\Lrect}{\mathsf{L}}
\DeclareMathOperator{\Krect}{\mathsf{K}}

\DeclareMathOperator{\ACof}{\mathsf{ACof}}
\DeclareMathOperator{\AFib}{\mathsf{AFib}}
\DeclareMathOperator{\Cof}{\mathsf{Cof}}
\DeclareMathOperator{\Fib}{\mathsf{Fib}}

\DeclareMathOperator{\Rfrak}{\mathfrak{R}}

\DeclareMathOperator{\Ho}{\mathsf{Ho}}

\DeclareMathOperator{\Eq}{\mathsf{Eq}}
\DeclareMathOperator{\Map}{\mathsf{Map}}

\DeclareMathOperator{\Hom}{\mathsf{Hom}}

\DeclareMathOperator{\Ch}{\mathsf{Ch}}

\DeclareMathOperator{\Spt}{\EuScript{S}\mathsf{pt}}

\usepackage{titlesec}

\usepackage{float}

\usepackage{fancyhdr}
\usepackage{tikz}
\usetikzlibrary{calc,decorations.pathmorphing,shapes}

 \newtheorem{thm}{Theorem}[section]
\newtheorem{taggedtheoremx}{Theorem}
\newenvironment{taggedtheorem}[1]
{\renewcommand\thetaggedtheoremx{#1}\taggedtheoremx}
{\endtaggedtheoremx}
\newtheorem{cor}[thm]{Corollary}
\newtheorem{lem}[thm]{Lemma}
\newtheorem{prop}[thm]{Proposition}

\theoremstyle{definition}
\newtheorem{defn}[thm]{Definition}
\newtheorem{notat}[thm]{Notation}

\theoremstyle{remark}
\newtheorem{rem}[thm]{Remark}
\newtheorem{rems}[thm]{Remarks}

 \title{
 	\begin{Large}
 		\textsc{When Bousfield localizations and homotopy idempotent functors meet again}	
 	\end{Large}
 }
 \author{Victor Carmona\thanks{The author was partially supported by the Spanish Ministry of Economy under the grant MTM2016-76453-C2-1-P (AEI/FEDER, UE), by the Andalusian Ministry of Economy and Knowledge and the Operational Program FEDER 2014-2020 under the grant US-1263032, by grant PID2020-117971GB-C21 of the Spanish Ministry of Science and Innovation, and grant FQM-213 of the Junta de Andaluc\'ia. He was also partly supported by Spanish Ministry of Science, Innovation and Universities grant FPU17/01871.}
 \date{September 2022}	
 }

\begin{document}

 	\maketitle
 	\vspace*{1mm}
 	\begin{abstract}
 		We generalize  Bousfield-Friedlander's Theorem and Hirschhorn's Localization Theorem 
 	to settings where the hypotheses are not satisfied at the expend of obtaining semimodel categories instead of model categories. We use such results to answer, in the world of semimodel categories, an open problem posed by May-Ponto about the existence of Bousfield localizations for Hurewicz and mixed type model structures (on spaces and chain complexes). We also provide new applications that were not available before, e.g.\ stabilization of non-cofibrantly generated model structures or applications to mathematical physics. \vspace*{5mm}	
 	\end{abstract}
 
 \maketitle

\tableofcontents

\section{Introduction}
	 The author's work on model structures for algebraic quantum field theories \cite{carmona_algebraic_2021} and factorization algebras (joint with Flores-Muro) \cite{carmona_model_2021} requires Bousfield localization techniques that were not available in the literature (see Section \ref{sect_Applications}). The essential obstacle is that the vast majority of known tools requires some properness asumption and/or smallness conditions. The present document provides general results to avoid some of these conditions at the expense of working with ``semi" instead of ``full" model categories.\footnote{Semimodel categories are a mild generalization of model categories which relax some of the factorization and lifting axioms asking them only for maps with cofibrant source or fibrant target. See \cite{fresse_modules_2009}.} 

The first tool is a generalization of the famous  Bousfield-Friedlander's theorem \cite{bousfield_homotopy_1978}. The underlying idea in the original result is that a nice homotopy idempotent functor\footnote{For the sake of conciseness, we refer to them as BF-(co)reflectors in the body of the paper.} on a nice model category can be used to construct a Bousfield localization. Our contribution is that working with semimodel categories instead of full model categories enables one to always construct the associated Bousfield localization for any  homotopy idempotent functor (and the process can be iterated).\footnote{There are subtleties concerning the notion of semimodel category that one has to use.} 
\begin{taggedtheorem}{\textbf{A}}[Theorem \ref{thm_LBLfromLocalizator}] \label{thm_A_LocalizatorLBL}
	Let $\M$ be a (right semi)model category and $\id\Rightarrow\Lrect$ a pointed homotopy idempotent functor. Then, $\M$ admits a left Bousfield localization at $\Lrect$-equivalences among right semimodel categories.
\end{taggedtheorem}

To the best of our knowledge, Theorem \ref{thm_A_LocalizatorLBL} is the first result that provides left Bousfield localizations of right semimodel categories and hence a good supply of non-trivial examples. Previously known existence theorems of left Bousfield localizations apply to left semimodel categories only, see \emph{Relation to other works} below for more.


Our second tool concerns a generalization of Hirschhorn's localization of cellular model structures, \cite[Theorem 4.1.1]{hirschhorn_model_2003}, in the absence of left properness (both the enriched and the non-enriched versions). 

\begin{taggedtheorem}{\textbf{B}}[Theorem \ref{thm_BousfieldLocalizationCellularCase}]\label{thm_B_CellularLBL} Let $\M$ be a cellular (left semi)model category whose generating sets have cofibrant domains and $\Srect$ a set of maps in $\M$. Then, the left Bousfield localization of $\M$ at $\Srect$ exists among left semimodel categories.
\end{taggedtheorem}  

Note that the working assumption that replaces left properness in Theorem \ref{thm_B_CellularLBL} is the cofibrancy of domains of generating (acyclic) cofibrations.

This work also collects several somehow surprising applications of Theorems \ref{thm_A_LocalizatorLBL} and \ref{thm_B_CellularLBL}.

We answer, in the realm of semimodel categories, an open question of May and Ponto in \cite{may_more_2012} that asks for the existence of Bousfield localizations at sets of maps for the Hurewicz and the mixed model structure on topological spaces (see \cite[Remark 19.5.7]{may_more_2012}). 

\begin{taggedtheorem}{\textbf{C}}[Propositions \ref{prop_LocalizationsOfHurewiczMixedModelsSIMPLICIALCASE} and \ref{prop_LocalizationsOfHurewiczMixedModelsTOPOLOGICALCASE}]
The Hurewicz and the mixed model structure on the category of topological spaces admit left Bousfield localizations at any set of maps among right semimodel categories.
\end{taggedtheorem} 

In fact, we distinguish between two ways to localize these model structures at a set of maps, we discuss right Bousfield localizations and our methods also apply to Hurewicz and mixed model structures on chain complexes.

Other applications of the main results, apart from the motivating examples coming from factorization algebras and algebraic quantum field theories, include: a criteria for Bousfield localization at a set of maps  assuming neither cellularity nor combinatoriality in Theorem \ref{thm_LBLatSequivalencesWith1StepLocalizator} or a stabilization machine for non-cofibrantly generated model categories (e.g. Hurewicz and mixed model structures).  See Section \ref{sect_Applications} for more.

\paragraph{Relation to other works.} For nice introductions to (homotopy) idempotent functors and Bousfield localizations, together with a plethora of examples where such techniques can be applied, see  \cite{dwyer_localizations_2004,lawson_introduction_2020}. See also \cite{balchin_handbook_2021} for all kind of data about model categories. 

Bousfield-Friedlander's Theorem originally appeared in \cite{bousfield_homotopy_1978}. Enhancements of this result are given in \cite{biedermann_duality_2015, bousfield_telescopic_2001, stanculescu_note_2008}. In particular, Proposition \ref{prop_LBLForNonFunctorialBFReflectors} generalizes Theorem \ref{thm_A_LocalizatorLBL} to incorporate Biedermann-Chorny's idea \cite[Appendix A]{biedermann_duality_2015}.

Regarding left and right Bousfield localizations of cellular model categories, the principal reference is Hirschhorn's book \cite{hirschhorn_model_2003}. In fact, our exposition and results owes a lot to this work. Hirschhorn's general localization results, together with Smith's Theorem \cite[Proposition 2.2]{barwick_left_2010}, are the most applied techniques in the literature concerning Bousfield localizations. However, both approaches require properness assumptions to work. Fortunately, several people observed that Bousfield localizations exist in some cases as semimodel categories without assuming properness assumptions and they deduced important results from this fact.    

When lacking left properness, the existence of left Bousfield localization for combinatorial (left semi)model categories is addressed in \cite{white_left_2020,henry_combinatorial_2023}. Indeed, Batanin-White \cite[Theorem B]{white_left_2020} generalize Smith's Theorem to combinatorial left semimodel categories and they  provide a thorough list of possible applications. It is also important to remark that substituting cellular by combinatorial in Theorem \ref{thm_B_CellularLBL}, one arrives at \cite[Theorem A]{white_left_2020}. Regarding non left proper cellular model categories, Goerss-Hopkins \cite{goerss_moduli_2005} and Harper-Zhang \cite{harper_topological_2019} showed that certain homological localizations of operadic algebras exist. However, to the best of our knowledge, there is no general result providing a parallel to the standard reference \cite{hirschhorn_model_2003} about existence of left Bousfield localization for cellular model categories. Theorem \ref{thm_B_CellularLBL} fills this gap.

\paragraph{Outline.} Section \ref{sect_Preliminaries} contains basic definitions of semimodel categories and an original discussion of Quillen functors and bifunctors between them. This is needed in order to compare left and right semimodel categories (Proposition \ref{prop_WeakQuillenPairsInduceAdjunctions}) and to have a sound theory of enrichment for semimodel categories (Proposition \ref{prop_CoreQuillenAdjOfTwoVariablesYieldAdjunctionsOfTwoVariables} and Remark \ref{rem_EnrichmentOfSemimodels}). Theorem \ref{thm_A_LocalizatorLBL} and its proof constitutes Section \ref{sect_LBLviaLocalizators}. The bridge between Theorems \ref{thm_A_LocalizatorLBL} and \ref{thm_B_CellularLBL} is constructed in Section \ref{sect_LocalizatorsFromSmallData}, where we try to obtain a pointed homotopy idempotent functor which corresponds to left Bousfield localization at a set of maps (Theorem \ref{thm_LBLatSequivalencesWith1StepLocalizator}). Left Bousfield localization of non left-proper cellular model categories is addressed in Section \ref{sect_LBLviaCellularity} and it leads up to Theorem \ref{thm_B_CellularLBL}. Finally, Section \ref{sect_Applications} is comprised of examples and applications of the main results.

\paragraph{Conventions:} We will freely use reduced notations for basic notions in the theory of model categories. For instance, \emph{llp} (resp. \emph{rlp}) refers to left (resp. right) lifting property, $\upi\boxslash\upp$ means $\upi$ satisfies llp with respect to $\upp$, \emph{wfs} means weak factorization system, \emph{soa} denotes small object argument, etc.

\section{Basics on semimodel structures}\label{sect_Preliminaries}
	In this section we review basic definitions of the theory of semimodel categories. The notions of coSpitzweck semimodel structure and of core Quillen adjunction are introduced. We will use notation and basic results from \cite{barwick_left_2010, fresse_modules_2009, henry_combinatorial_2023}.

\paragraph{Semimodel categories.}

\begin{defn} A \emph{structured homotopical category} $\M$ is a category $\Mrect$ endowed with three classes of maps $(\Cof,\Eq,\Fib)$ which satisfy:
	\begin{itemize}
		\item $\Mrect$ is bicomplete, i.e. has all limits and colimits.
		\item $\Eq$ is closed under retracts and satisfies 2-out of-3.
		\item $\Cof$ (resp.\ $\Fib$) is closed under retracts and pushouts (resp.\ pullbacks).
	\end{itemize}
\end{defn}

For brevity we make use of the following notation introduced in \cite{henry_weak_2020,henry_combinatorial_2023}.
\begin{defn} A \emph{core cofibration} (resp.\ \emph{core fibration}) is a cofibration that has cofibrant source (resp.\ fibrations with fibrant target). We denote the class of core cofibrations (resp.\ core fibrations) by $\Cof_{\circ}$ (resp.\ $\Fib_{\circ}$).
\end{defn}

\begin{notat} We also adopt the following classical notation:
	$$
	\begin{array}{ccccc}
\ACof_{(\circ)} & \equiv & \{\text{(core) acyclic cofibrations}\} & \equiv & 	\Cof_{(\circ)}\cap\Eq\vspace*{2mm}\\
\AFib_{(\circ)}	 & \equiv & \{\text{(core) acyclic fibrations}\} & \equiv & \Fib_{(\circ)}\cap\Eq
	\end{array}.
	$$
\end{notat}

\begin{defn}\label{defn_FresseSemimodel} A \emph{Fresse right semimodel category} $\M$ is a structured homotopical category $(\Mrect,\Cof,\Eq,\Fib)$ which satisfies:
	\begin{itemize}
		\item[(i)] (Lifting axioms) $\ACof\boxslash\Fib_{\circ}$ and $\Cof\boxslash\AFib_{\circ}$.
		\item[(ii)] (Factorization axioms) Any map $\upf$ with fibrant target can be factored as
		$$
		\upf\colon\bullet\xrightarrow{\ACof}\bullet\xrightarrow{\Fib_{\circ}}\bullet \quad \text{ or as }\quad\upf\colon\bullet\xrightarrow{\Cof}\bullet\xrightarrow{\AFib_{\circ}}\bullet.
		$$
		\item[(iii)] the terminal object $\mathbb{1}$ is fibrant.
	\end{itemize}
\end{defn}

\begin{defn}\label{defn_SpitzweckSemimodel} A \emph{Spitzweck right semimodel category} is a Fresse right semimodel category $\M$ which additionally satisfies that $(\ACof,\Fib)$ is a weak factorization system (wfs), i.e. 
	\item[(i)] (Lifting axiom) $\ACof\boxslash\Fib$.
	\item[(ii)] (Factorization axiom) Any map $\upf$ can be factored as 
	$
	\upf\colon\bullet\xrightarrow{\ACof}\bullet\xrightarrow{\Fib}\bullet.
	$
\end{defn}
	
\begin{defn}\label{defn_coSpitzweckSemimodel}	
	A \emph{coSpitzweck right semimodel category} is a Fresse right semimodel category $\M$ which additionally satisfies that $(\Cof,\AFib)$ is a wfs.
\end{defn}

 Of course, left semimodel categories can also be defined. The more concise way to do so is by appealing to duality: $\M=(\Mrect,\Cof,\Eq,\Fib)$ is a \emph{Fresse} (resp. \emph{(co)Spitzweck}) \emph{left semimodel} category iff $\M^{\text{op}}=(\Mrect^{\text{op}},\Fib^{\text{op}},\Eq^{\text{op}},\Cof^{\text{op}})$ is a Fresse (resp. (co)Spitzweck) right semimodel category.

\begin{rem} A condensed way to record what the axioms for the three variants of right semimodel structures are is given by the following diagram
	$$
	
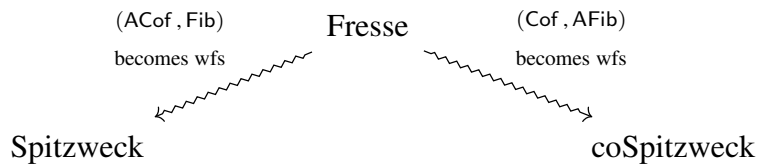
\begin{tikzcd}[ampersand replacement=\&]
	\& \& \begin{matrix}
	\text{Fresse} 
	\end{matrix}\ar[lld, "\begin{matrix}
	(\ACof\text{,}\Fib)\\
	\text{becomes wfs}
	\end{matrix}"', rightsquigarrow] \ar[rrd, "\begin{matrix}
	(\Cof\text{,}\AFib)\\
	\text{becomes wfs}
	\end{matrix}", rightsquigarrow]\& \&\\
	\begin{matrix}
	\text{Spitzweck} 
	\end{matrix} \& \& \& \& \begin{matrix}
	\text{coSpitzweck}
	\end{matrix}
	\end{tikzcd}.
	$$
	It is also useful to consider that $(\Cof,\AFib)$ and $(\ACof,\Fib)$ in a Fresse right semimodel structure are not wfs, but a sort of \emph{right wfs}. Thus, passing from Fresse to (co)Spitzweck semimodel structures asks for an enhancement of one of those right wfs to a wfs.
\end{rem}

 \begin{notat}
In the sequel, by semimodel category we understand any of the variants: left or right Fresse/(co)Spitzweck semimodel structure. If we also specify a direction, e.g. left semimodel category, we want to refer to any variant of left semimodel structure.
 \end{notat}

\begin{rems}\label{rems_SemimodelsAndTractability} A couple of simple observations:
	\begin{itemize}
		\item There are even weaker variants of semimodel categories whose lifting axioms are related to another structured homotopical category via an adjunction, see \cite[Section 1]{barwick_left_2010} and \cite[Subsection 12.1.9]{fresse_modules_2009}.
		\item There is a natural notion of cofibrantly generated left or right semimodel structure (\cite[Definition 1.21]{barwick_left_2010} and \cite[Subsection 12.1.3]{fresse_modules_2009}).
		\item It seems that there is not a clear consensus on what \emph{tractable} means for (semi)model categories. In contrast with \cite{barwick_left_2010}, we call a (semi)model structure  \emph{tractable} if it is cofibrantly generated and its generating sets can be chosen to have cofibrant domains (they are core cofibrations).
	\end{itemize}	
\end{rems}

\paragraph{(Core) Quillen functors and adjunctions.}
The definition of Quillen adjunction for semimodel categories goes as follows.
\begin{defn}\label{defn_QuillenPairOfSemimodels}
A \emph{Quillen adjunction} between right semimodel categories is an ordinary adjunction $\Frect\dashv\Urect$ such that $\Frect$ preserves (acyclic) cofibrations. 
\end{defn}
\begin{rem} In contrast with full model structures, Definition \ref{defn_QuillenPairOfSemimodels} does not admit reformulations in terms of the right adjoint or in terms of both functors together like Quillen's original definition of Quillen adjunction. Consequently, we can not define a right Quillen functor between right semimodel categories without appealing to its left adjoint. Besides, it is clear that this notion cannot be used to connect left and right semimodel categories.
	
By the axioms of a Fresse right semimodel structure, the right adjoint of a Quillen adjunction is only known to preserve core (acyclic) fibrations. That suffices to get an adjunction between homotopy categories (see Proposition \ref{prop_WeakQuillenPairsInduceAdjunctions}).
\end{rem}

Motivated by these issues, we define a mild generalization which is more symmetric and which suffices to produce adjunctions at a homotopical level.

\begin{defn}\label{defn_WeakQuillenPairOfSemimodels}
	A \emph{core Quillen adjunction} between semimodel categories is an ordinary adjunction $\Frect\dashv\Urect$ such that $\Frect$ preserves core (acyclic) cofibrations and $\Urect$ preserves core (acyclic) fibrations. 
\end{defn}

It is clear that Quillen adjunctions are core Quillen adjunctions by the remark above, but there is also a converse statement for tractable semimodel categories.

\begin{prop}\label{prop_WeakQuillenPairForCoreCofibGen}
	Let $\Frect\colon\M\rightleftarrows\OpN\colon\Urect$ be a plain adjunction between semimodel categories. If $\M$ is tractable and $\Frect$ preserves core (acyclic) cofibrations, then $\Urect$ preserves (acyclic) fibrations. In particular, if $\Frect\dashv\Urect$ is a core Quillen adjunction, it is also a Quillen adjunction.
\end{prop}
   \begin{proof}
   It is immediate from $\M$ being tractable by transporting lifting problems along the adjunction, i.e. $
   \Frect(\upi)\boxslash\upp$ iff $\upi\boxslash\Urect(\upp)$.
   \end{proof}

The following lemmas are designed to prove that a core Quillen adjunction between semimodel categories induces an adjunction between homotopy categories (Proposition \ref{prop_WeakQuillenPairsInduceAdjunctions}).

\begin{lem}\label{lem_LeftHomotopyClassesIdentifyCofibrantReplacements}
If $\M$ is a semimodel category and $\Xrect$ is a bifibrant object in $\M$, the hom-functor modulo left homotopies, $\Zrect\mapsto \Mrect(\Xrect,\Zrect)/\hspace*{-1mm}\sim_{l}\,$, where $\Zrect$ varies over fibrant objects, 
sends core acyclic fibrations to bijections. The dual statement also holds.
\end{lem}
\begin{proof}	
 To check the claim, consider a core acyclic fibration $\upp\colon\Zrect\to\Trect$. Regardless of wether $\M$ is a left or a right semimodel category, $\upp_*\colon\Mrect(\Xrect,\Zrect)\to\Mrect(\Xrect,\Trect)$ is surjective (and so it is also surjective when taking quotients by left homotopy relation). It remains to see that $\upp_*$ is also injective if we consider left homotopy classes of maps. This follows by a simple lifting argument.
\end{proof}

\begin{lem}\label{lem_LeftAndRightHomotopiesCoincideSometimes}
	Let $\M$ be a semimodel category, $\Xrect$ a bifibrant object and $\Zrect$ a fibrant object. If $\Zrect$ admits a good path object, i.e. a factorization of the diagonal map $\Zrect\to\Zrect\times\Zrect$ as an equivalence followed by a fibration, then the notions of left and right homotopy coincide in the set $\Mrect(\Xrect,\Zrect)$. The dual statement also holds. 
\end{lem}
\begin{proof}
	See the proof of  \cite[Proposition 1.2.5 (v)]{hovey_model_1999}; the cited hypotheses are enough for that argument to work.
\end{proof}

\begin{prop}\label{prop_WeakQuillenPairsInduceAdjunctions} 
If $\Frect\dashv\Urect$ is a core Quillen adjunction between semimodel categories, then it induces an adjunction $\mathbb{L}\Frect\dashv\mathbb{R}\Urect$ between homotopy categories.
\end{prop}
\begin{proof}
 By Ken Brown's Lemma (which obviously works for semimodel categories), $\Frect\cdot\mathsf{Rep}_{\text{cf}}$ induces a functor $\mathbb{L}\Frect\colon\Ho\M\to\Ho\OpN$, where $\mathsf{Rep}_{\text{cf}}$ denotes a bifibrant replacement functor. Analogously one obtains $\mathbb{R}\Urect\colon\Ho\OpN\to\Ho\M$. What remains is to check that $\mathbb{L}\Frect$ and $\mathbb{R}\Urect$ are adjoint. The natural bijection between hom-sets  $\Ho\OpN(\mathbb{L}\Frect(\Xrect),\Yrect)\cong\Ho\M(\Xrect,\mathbb{R}\Urect(\Yrect))$ is more clearly described assuming $\Xrect$ and $\Yrect$ are already bifibrant. In this case it is given by the following composition of natural bijections, where $\mathsf{R}$ and $\mathsf{Q}$ represent fibrant and cofibrant replacements respectively, 
$$
\begin{tikzcd}[ampersand replacement=\&]
\Ho\OpN(\Frect(\Xrect),\Yrect) \ar[d,"(i)"']\& \& \&  \Ho\M(\Xrect,\Urect(\Yrect))\\
\underset{\text{right homotopy}}{\underline{\Nrect(\mathsf{R}\Frect(\Xrect),\Yrect)}}\ar[r, "(ii)"]\&
 \underset{\text{right homotopy}}{\underline{\Nrect(\Frect(\Xrect),\Yrect)}}\ar[r, "(iii)"] \& \underset{\text{left homotopy}}{\underline{\Mrect(\Xrect,\Urect(\Yrect))}}\ar[r, "(iv)"] \& \underset{\text{left homotopy}}{\underline{\Mrect(\Xrect,\mathsf{Q}\Urect(\Yrect))}} \ar[u, "(v)"']
\end{tikzcd}.
$$
This is the case due to the following facts: ($i$) in a  semimodel category, maps in the homotopy category can be given by left homotopy classes of maps between bifibrant objects. ($ii$) observe that $\Yrect$ is bifibrant and apply the dual of Lemma \ref{lem_LeftHomotopyClassesIdentifyCofibrantReplacements}. ($iii$) on the left hand side, we can interchange left and right homotopy relations by Lemma \ref{lem_LeftAndRightHomotopiesCoincideSometimes}. When doing so, just note that the adjunction $\Frect\dashv\Urect$ respects the left homotopy relation (since $\Frect$ preserves cylinders for cofibrant objects) and so induces a bijection between quotients. 
($iv$) observe that $\Xrect$ is bifibrant and apply Lemma \ref{lem_LeftHomotopyClassesIdentifyCofibrantReplacements}. ($v$) same as ($i$).
\end{proof}


\paragraph{Core Quillen bifunctors and enrichment.} As for Quillen functors, the theory of Quillen bifunctors (or Quillen functors with more variables) does not work well for semimodel categories. This fact produces complications if one is willing to deal with enriched semimodel categories, monoidal semimodel categories, etc.

\begin{rem}\label{rem_EnrichmentOfSemimodels}
 For instance, a little bit of thought about how a bifunctor $\otimes\colon\Mrect\times\Mrect\to \Mrect$ can be compatible with a left semimodel structure on $\Mrect$ shows several obstructions.
\end{rem}

 For these reasons, we propose the obvious extension of core Quillen functors: 

\begin{defn}\label{defn_CoreQBifunctor} A \emph{core left Quillen bifunctor} is a bifunctor 
	$\boxdot\colon\Mrect\times\Nrect\rightarrow\Trect$
	between semimodel categories which satisfies the pushout product axiom just for core cofibrations. In other words, the associated pushout product functor between arrow categories $\square\colon \Mrect^{\mathbb{2}}\times\Nrect^{\mathbb{2}}\to\Trect^{\mathbb{2}}$ sends a pair of core cofibrations to a core cofibration which is acyclic if one of the original core cofibrations is acyclic.
	
	A \emph{core Quillen adjunction of two variables} is a two variable adjunction  $(\Map_{\text{l}},\boxdot,\Map_{\text{r}})$ such that $\boxdot$ is core left Quillen and $\Map_{\text{l/r}}$ are core right Quillen. 
\end{defn}

A simple generalization of the arguments given for core Quillen adjunctions leads us to:
\begin{prop}\label{prop_CoreQuillenAdjOfTwoVariablesYieldAdjunctionsOfTwoVariables} A core Quillen adjunction of two variables between semimodel categories induces an adjunction of two variables between homotopy categories.
\end{prop}

The definition of enriched (resp. closed symmetric monoidal...) semimodel category is natural from here. 

\begin{defn}\label{defn_VEnrichedSemimodel} Let $\V$ and $\M$ be semimodel categories.
	\begin{itemize}
		\item We say that $\V$ is a \emph{closed symmetric monoidal} semimodel category if it has a closed symmetric monoidal structure which is a core Quillen adjunction of two variables satisfying the unit axioms (\cite[Definition 4.2.6 (2) and Proposition 4.2.7]{hovey_model_1999}).
		\item Consider that $\V$ has the structure of the previous item. We say that $\M$ is a \emph{$\V$-semimodel category} if it is $\V$-tensored/cotensored/enriched, this structure establishes a core Quillen adjunction of two variables and those bifunctors also satisfy the appropiate unit axiom, i.e.\ the map $\Xrect\otimes \Qrect\mathbb{I}\to \Xrect$
		induced by a cofibrant approximation of the unit  $\Qrect\mathbb{I}\to\mathbb{I}$
		is a weak equivalence for all cofibrant objects $\Xrect\in\M$.
	\end{itemize}
\end{defn}

Once the apparatus of semimodel structures is settle, we turn into Bousfield localizations.

\section{Bousfield localization \`a la Bousfield-Friedlander}\label{sect_LBLviaLocalizators}
	
The most elegant way to look at localizations is through their universal property.
\begin{defn}\label{defn_LBL} Let $\M$ be a model category and $\Wrect$ a class of maps in $\M$. A left Bousfield localization of $\M$ at $\Wrect$ is a model $\mathsf{L}_{\Wrect}\M$ equipped with a left Quillen functor $\M\to\mathsf{L}_{\Wrect}\M$ which is initial among left Quillen functors $\M\to\OpN$ that send $\Wrect$ to equivalences.
\end{defn}	
\begin{rem} Of course, this universal property could be modified in several directions. For instance, one can ask that $\mathsf{L}_{\Wrect}\M$ has only a left (resp. right) semimodel structure or that the left Quillen functor $\M\to\mathsf{L}_{\Wrect}\M$ is only the left adjoint in a core Quillen adjunction. On a different direction, one can ask that $\M$ has some enrichment or monoidal structure and that so has the localization.
	
Instead of introducing the zoo of different modifications that Definition \ref{defn_LBL} admits, we leave the reader this task and we will just provide enough information when constructing Bousfield localizations to identify which is exactly the strongest universal property that it enjoys.
\end{rem}

\begin{paragraph}{Left Bousfield localization of right semimodel structures.}
	   Relying on Bousfield-Friedlander theorem (see \cite[Appendix A]{bousfield_homotopy_1978}), we make the following definition.
  \begin{defn}\label{defn_BFReflector}
  Let $\M$ be a semimodel category. A \emph{BF-reflector} on $\M$ is an endofunctor $\Rfrak\colon\Mrect\to\Mrect$ with a natural transformations
  $
  \upeta\colon \id\Rightarrow\Rfrak
  $
  satisfying: ($i$) $\Rfrak$ is homotopical and ($ii$) $\Rfrak\upeta$ and $\upeta\Rfrak$ are natural equivalences. 
  \end{defn}
  
  We just name BF-reflectors after Bousfield and Friedlander, since they are the essential ingredient in \cite[Theorem A.7]{bousfield_homotopy_1978}, but they have been also called \emph{pointed homotopy idempotent functors} in the literature.  Note that Bousfield-Friedlander asked BF-reflectors to satisfy an additional property, (A.6) in loc.cit.\;(which is hard to check in practice), to make their result work. Our goal in this section is to show that a choice of BF-reflector in $\M$ is enough to construct a left Bousfield localization of $\M$. 
  
  Our goal parallels \cite[Theorem A.7]{bousfield_homotopy_1978} (or more precisely Stanculescu's generalization of it  \cite{stanculescu_note_2008}) and the proofs follow the same lines. The fundamental new idea is similar to Goerss-Hopkins's insight to construct a left Bousfield localization without asking for left properness (see \cite[Section 1.5]{goerss_moduli_2005}).

  It will be useful to introduce a bit of notation.
 \begin{defn} Let $\Rfrak$ be a BF-reflector on a right semimodel category $\M$. 
 	\begin{itemize}
 		\item  An $\Rfrak$-equivalence is a map $\upf$ in $\Mrect$ such that $\Rfrak(\upf)$ is an equivalence. The class of all $\Rfrak$-equivalences is denoted $\Eq_{\Rfrak}=\Rfrak^{-1}(\Eq)$.
 		\item A $\Rfrak$-local object is an object $\Xrect\in\Mrect$ for which $\upeta_{\Xrect}$ is an equivalence. 
 		\item  A $\Rfrak$-\emph{fibration} is a map with the right lifting property with respect to all cofibrations which are $\Rfrak$-equivalences. The class of $\Rfrak$-fibrations is denoted $\Fib_{\Rfrak}=\big(\Cof\cap\Eq_{\Rfrak}\big)^{\boxslash}$.
 	\end{itemize}
 \end{defn}

With this notation at hand, we are ready to state the main result in this section:
\begin{thm}\label{thm_LBLfromLocalizator} Let $\M=(\Mrect,\Cof,\Eq,\Fib)$ be a Fresse (resp. coSpitzweck) right semimodel category and $\Rfrak$ a BF-reflector on $\M$. Then, 
	$
	\M_{\Rfrak}=(\Mrect,\Cof,\Eq_{\Rfrak},\Fib_{\Rfrak})
	$
is the left Bousfield localization of $\M$ at $\Rfrak$-equivalences among  Fresse (resp. coSpitzweck) right semimodel categories with left Quillen functors between them.
\end{thm}

Note that the universal property in Definition \ref{defn_LBL} is automatically satisfied.

We will discuss each axiom of semimodel category separately, but first we collect the fundamental facts that made the theorem work.

Since we have defined $\Fib_{\Rfrak}$ as those maps that have right lifting property against $\Cof\cap\Eq_{\Rfrak}$, we need some additional control of core $\Rfrak$-fibrations. 

\begin{prop}\label{prop_ACofInLBLwithLocalizator} A cofibration is an $\Rfrak$-equivalence iff it has llp against core fibrations between $\Rfrak$-local objects.
\end{prop}
\begin{proof} The proofs of \cite[Lemmas 2.3 and 2.4]{stanculescu_note_2008} work for semimodel structures since the required replacements are constructed using $\big(\ACof,\Fib_{\circ}\big)$-factorizations of maps with fibrant target.
\end{proof}

 In the cited reference, the following classical lemma is not mentioned but it plays a fundamental role in \cite[Lemmas 2.3 and 2.4]{stanculescu_note_2008}. We also use it in the proof of Theorem \ref{thm_LBLfromLocalizator}.

\begin{lem}\label{lem_Hopullbacks} In a right semimodel category, consider a pullback square
	$$
	\begin{tikzcd}[ampersand replacement=\&]
	\Xrect'\ar[r, "\upf'"]\ar[d]\ar[rd, phantom, "\lrcorner"] \& \Yrect'\ar[d, twoheadrightarrow]\\
	\Xrect\ar[r, "\upf"'] \& \Yrect
	\end{tikzcd}
	$$
	between fibrant objects. Then, if $\upf$ is an equivalence, its base change $\upf'$ is so. 
\end{lem}
\begin{proof}
See the proof of \cite[Proposition 1.19]{barwick_left_2010}.
\end{proof}

\begin{prop}\label{prop_CharacterizationOfFibrantsInLBLwithLocalizator} An object is $\Rfrak$-fibrant iff it is $\Rfrak$-local and fibrant.
\end{prop}
\begin{proof}
	The ``if" direction is a consequence of Proposition \ref{prop_ACofInLBLwithLocalizator}. For the sufficiency, note that a $\Rfrak$-fibrant object $\Xrect$ has the rlp against $\Cof\cap\Eq_{\Rfrak}$. Thus, $\Xrect$ is a retract of a fibrant replacement of $\Rfrak\Xrect$, and so it is $\Rfrak$-local and fibrant.
\end{proof}

\begin{prop}\label{prop_CharacterizationOfCoreFibrationsInLBLwithLocalizator} A map is a core $\Rfrak$-fibration iff it is a core fibration between $\Rfrak$-local objects.
\end{prop}
\begin{proof}
	The ``if" direction is a direct consequence of Propositions \ref{prop_ACofInLBLwithLocalizator} and \ref{prop_CharacterizationOfFibrantsInLBLwithLocalizator}. For the sufficiency, show that every core $\Rfrak$-fibration is a retract of a core fibration by factoring it into an acyclic cofibration followed by a core fibration. 
\end{proof}

These preliminary results put us in place for the

\begin{proof}[Proof of Theorem \ref{thm_LBLfromLocalizator}.]$\,$
	
\textit{Underlying structured homotopical category.} By direct inspection.

\textit{Lifting axioms.} By definition $\Cof\cap\Eq_{\Rfrak}\boxslash\Fib_{\Rfrak}$. It remains to check the other lifting axiom: $\Cof\boxslash\Fib_{\Rfrak,\circ}\cap\Eq_{\Rfrak}$. Let us assume first that $\M$ is Fresse right semimodel category. Since we know that $\Cof\boxslash\AFib_{\circ}$, we can obtain the required lifting by showing that we have an inclusion  $\Fib_{\Rfrak,\circ}\cap\Eq_{\Rfrak}\subseteq\AFib_{\circ}$. As usual, this can be proven by factorizing as $(\Cof,\AFib_{\circ})$, which is allowed by Proposition \ref{prop_CharacterizationOfFibrantsInLBLwithLocalizator}, and applying the retract argument.

When $\M$ is coSpitzweck right semimodel category, one has even an equality between the classes  $\Fib_{\Rfrak}\cap\Eq_{\Rfrak}=\AFib$ without core assumptions by the exact same argument.

\textit{Factorization axioms.}
 Given a map $\upg\colon\Xrect\to\Yrect$ with $\Rfrak$-fibrant target, we should factor it in two ways:
\begin{enumerate}
	\item[(i)]  $\upg=\upp\cdot\upj$ where $\upj$ belongs to $\Cof$ and $\upp$ belongs to $\Fib_{\Rfrak,\circ}\cap\Eq_{\Rfrak}$. Just use the factorization $\big(\Cof,\AFib_{\circ}\big)$ and apply the  characterization of $\Rfrak$-fibrant objects and core $\Rfrak$-fibrations (Propositions \ref{prop_CharacterizationOfFibrantsInLBLwithLocalizator} and \ref{prop_CharacterizationOfCoreFibrationsInLBLwithLocalizator}). In coSpitzweck's case, this factorization exists without hypothesis on the target and we apply the identity $\Fib_{\Rfrak}\cap\Eq_{\Rfrak}=\AFib$ to deduce the axiom.
	\item[(ii)]  $\upg=\upp\cdot\upj$ where $\upj$ belongs to $\Cof\cap\Eq_{\Rfrak}$ and $\upp$ belongs to $\Fib_{\Rfrak,\circ}$. Note that the BF-reflector $(\Rfrak,\upeta)$ yields a commutative square
	$$
	\begin{tikzcd}[ampersand replacement=\&]
	\Xrect \ar[d, "\upg"'] \ar[r] \& \Rfrak\Xrect\ar[d,"\Rfrak\upg"]\; \\
    \Yrect \ar[r, "\upeta"'] \& \Rfrak\Yrect
	\end{tikzcd}.       
	$$
	Up to taking $\big(\ACof,\Fib_{\circ}\big)$-factorizations, we can assume that $\Rfrak\upg$ is a fibration between $\Rfrak$-local and fibrant objects. By Proposition \ref{prop_CharacterizationOfCoreFibrationsInLBLwithLocalizator}, $\Rfrak\upg$ is a core $\Rfrak$-fibration, so pulling it back along $\upeta$ gives us a commutative diagram
	$$
	\begin{tikzcd}[ampersand replacement=\&]
	\Xrect \ar[r]\ar[d, "\upg"'] \& \Xrect' \ar[d, "\uph"'] \ar[r, "\upeta'"] \ar[rd, phantom, "\lrcorner"]\& \Rfrak\Xrect\ar[d,"\Rfrak\upg"]\; \\
	\Yrect \ar[r, equal] \& \Yrect \ar[r, "\upeta"'] \& \Rfrak\Yrect
	\end{tikzcd},
	$$
	where $\uph\colon \Xrect'\to\Yrect$ is a core $\Rfrak$-fibration. Note also that the top map in the left square  $\Xrect\to\Xrect'$ is an $\Rfrak$-equivalence by 2-out of-3 since $\upeta'\colon\Xrect'\to\Rfrak\Xrect$ is an equivalence (see Lemma \ref{lem_Hopullbacks}). Finally, factor $\Xrect\to\Xrect'$ as $\upq\cdot\upj$ 
	with $\upj\in\Cof$ and $\upq\in\AFib_{\circ}$ using that $\Xrect'$ is fibrant. Hence, we obtain the factorization  $\upg=(\uph\cdot\upq)\cdot\upj$ where $\upj$ belongs to $\Cof\cap\Eq_{\Rfrak}$ by 2-out of-3 and $(\uph\cdot\upq)$ belongs to $\Fib_{\Rfrak,\circ}$ by Proposition \ref{prop_CharacterizationOfCoreFibrationsInLBLwithLocalizator}. 
\end{enumerate}

\textit{Fibrancy of terminal object.}
 By the characterization of $\Rfrak$-fibrant objects (Proposition \ref{prop_CharacterizationOfFibrantsInLBLwithLocalizator}), it remains to check that $\mathbb{1}$ is $\Rfrak$-local. This fact follows by observing that $\mathbb{1}$ is a retract of $\Rfrak\mathbb{1}$. 

\end{proof}

\begin{rem} Note that in particular we have established that
the fibrant objects and core fibrations of $\M_{\Rfrak}$ are characterized in terms of the BF-reflector $\Rfrak$ (Propositions \ref{prop_CharacterizationOfFibrantsInLBLwithLocalizator} and \ref{prop_CharacterizationOfCoreFibrationsInLBLwithLocalizator}).
\end{rem}

One open question raised by Theorem \ref{thm_LBLfromLocalizator} is if the Bousfield localization inherits additional structures or properties when the original (semi)model category has or satisfies them. We prove a result in this direction, Proposition \ref{prop_EnrichmentForLocalizator}, but first we need a preliminary lemma.
\begin{lem}\label{lem_CharacterizationOfRequivalencesWithEnrichment}
Under the hypotheses of Theorem \ref{thm_LBLfromLocalizator} and assuming  $\M$ is a  $\V$-semimodel category and $(\Rfrak,\upeta)$ is $\V$-enriched, a map $\upf$ is an $\Rfrak$-equivalence iff $\mathbb{R}\underline{\M}(\upf,\Xrect)$ is an equivalence for any $\Rfrak$-local object $\Xrect$, where $\underline{\M}\colon \M^{\text{op}}\times\M\to\V$ denotes the $\V$-enrichment of $\M$.
\end{lem}
\begin{proof}
	The direct implication is a consequence of the following (homotopy) commutative diagram and the retract closure of equivalences
	$$
    \begin{tikzcd}[ampersand replacement=\&]
    \id\colon\mathbb{R}\underline{\M}(\Arect,\Xrect)\ar[r, "\Rfrak"] 
    \ar[d,"\upf^*"']\& \mathbb{R}\underline{\M}(\Rfrak\Arect,\Rfrak\Xrect)\overset{\upeta_*}{\simeq}\mathbb{R}\underline{\M}(\Rfrak\Arect,\Xrect) \ar[r, "\upeta^*"]\ar[d,"\Rfrak\upf^*"', shift left=15] \& \mathbb{R}\underline{\M}(\Arect,\Xrect)\ar[d,"\upf^*"]\\
    \id\colon\mathbb{R}\underline{\M}(\Brect,\Xrect)\ar[r, "\Rfrak"']
     \& \mathbb{R}\underline{\M}(\Rfrak\Brect,\Rfrak\Xrect)\underset{\upeta_*}{\simeq}\mathbb{R}\underline{\M}(\Rfrak\Brect,\Xrect) \ar[r, "\upeta^*"'] \& \mathbb{R}\underline{\M}(\Brect,\Xrect)
    \end{tikzcd}.
	$$
	Note that $\Rfrak$ being homotopical and compatible with the $\V$-mapping object $\underline{\M}$, which is core Quillen bifunctor, makes possible to consider such diagram.
	
	For the sufficiency, we need to check that $\Rfrak\upf$ is an equivalence. By Yoneda Lemma, one can prove this by showing that $\Ho\M_{\Rfrak}(\Rfrak\upf,\Zrect)$ is a bijection for any fibrant  $\Zrect\in\M_{\Rfrak}$. Since the inclusion $\Ho\M_{\Rfrak}\hookrightarrow\Ho\M$ is fully faithful, one can compute these hom-sets in $\Ho\M$. Using the $\V$-enrichment of $\M$, it is easy to see that we are done if $\mathbb{R}\underline{\M}(\Rfrak\upf,\Zrect)$ is an equivalence for any fibrant  $\Zrect\in\M_{\Rfrak}$. Apply  Proposition \ref{prop_CharacterizationOfFibrantsInLBLwithLocalizator} to conclude the proof.
\end{proof}

\begin{prop}\label{prop_EnrichmentForLocalizator}
Under the hypotheses of Theorem \ref{thm_LBLfromLocalizator}, if $\M$ is a  $\V$-semimodel category and $(\Rfrak,\upeta)$ is $\V$-enriched, then $\M_{\Rfrak}$ is also a $\V$-semimodel category.
\end{prop}  
\begin{proof}
	We need to prove that $\V$-tensor/cotensor/enrichment of $\M$ stablishes a core Quillen adjunction of two variables for $\M_{\Rfrak}$. This claim is reduced to prove the acyclic part of the definition of core Quillen bifunctor for $\M_{\Rfrak}$ for each bifunctor.
	\begin{itemize}
		\item The $\V$-tensor $\odot$ is core Quillen: it suffices to check that the associated pushout product of two core cofibrations  $\upi\square\upj$ is an $\Rfrak$-equivalence when $\upj$ is so. Apply the $\V$-enrichment functor $\underline{\M}(\star,\Xrect)$, for $\Xrect$ fibrant $\Rfrak$-local, to the (homotopy) pushout diagram defining $\upi\square\upj$ and Lemma \ref{lem_CharacterizationOfRequivalencesWithEnrichment}. 
		
		\item The $\V$-cotensor $(\star)^{\bullet}$ is core Quillen: it suffices to check that the associated pullback hom $\widehat{\upp}^{\upi}$ of a core cofibration $\upi$ and a core fibration $\upp$ is an $\Rfrak$-equivalence when $\upp$ is so. Immediate by Proposition \ref{prop_CharacterizationOfCoreFibrationsInLBLwithLocalizator}.

		\item  The $\V$-enrichment $\underline{\M}$ is core Quillen: it suffices to check that the associated pullback hom $\widehat{\underline{\M}}(\upi,\upp)$ of a core cofibration $\upi$ and a core fibration $\upp$ is an equivalence when $\upi$ or $\upp$ is an $\Rfrak$-equivalence. When $\upp$ is an $\Rfrak$-equivalence, there is nothing to prove by Proposition \ref{prop_CharacterizationOfCoreFibrationsInLBLwithLocalizator}. When $\upi$ is the $\Rfrak$-equivalence, the claim follows from Lemma \ref{lem_CharacterizationOfRequivalencesWithEnrichment}.
	\end{itemize}
\end{proof}

\paragraph{Relaxing functoriality of BF-reflectors.} Just as a side remark, we observe that the functoriality assumed in the definition of BF-reflector can be relaxed a bit. In \cite[Appendix A]{biedermann_duality_2015}, it is proposed a way to do so that still yields a Bousfield localization of proper model categories. Inspecting their proof and our arguments, one observes that properness can be dropped in \cite[Theorem A.8]{biedermann_duality_2015} and that one can get the following result.
\begin{prop}\label{prop_LBLForNonFunctorialBFReflectors} The conclusion of Theorem \ref{thm_LBLfromLocalizator} holds if we replace the BF-reflector $\Rfrak$ by any homotopy localization construction satisfying A2-A5 in the sense of \cite[Appendix A]{biedermann_duality_2015}.
\end{prop} 

 \begin{rem} The results in this section are completely dualizable. That is, they are valid to construct right Bousfield localizations of left semimodel categories.
\end{rem}
\end{paragraph}

\section{Producing Bousfield localizations from small data}\label{sect_LocalizatorsFromSmallData}
This section deals with a modest generalization of an usual method to construct left Bousfield localizations in practice (see \cite[Section 4.2]{hirschhorn_model_2003}). This mild generalization will allow us to produce Bousfield localizations that were not known to exist so far in Section \ref{sect_Applications}.

\paragraph{Finding a BF-reflector from small data.}
	When one tries to understand the homotopy theory of a full subcategory of $\Ho\M$, it turns out that finding a set of maps $\Srect$ with respect to which the objects of the full subcategory are just $\Srect$-local objects (in some sense) is a fruitful idea. There are mainly two ways for an object to be $\Srect$-local: using the canonical mapping spaces associated to any (semi)model category (via frames) or using an enrichment of the (semi)model structure. In both situations, we require that the (local) object $\Zrect$ witnesses the maps of $\Srect$ as equivalences through the lenses of the selected derived mapping functor $\Map(\star,\Zrect)$. 

    We will focus on the enriched case since it is less common, but everything in this section can be adapted to the use of mapping spaces from framings. In order to accomodate both left and right semimodel structures in the same discussion, we assume that $\M$ is a full model category in this section, remarking the needed modifications for semimodel categories at the end.
	
    Given a set of maps $\Srect$ in $\M$, we want to find a BF-reflector $\Rfrak\colon\M\to \M$ to model the homotopy theory of $\Srect$-local objects. We see that two different paths appear depending on if we assume $\M$ to be left proper or not. If we assume that $\M$ is left proper, we will obtain a BF-reflector in great generality and so a Bousfield localization presenting the homotopy theory of $\Srect$-local objects (by Theorem \ref{thm_LBLfromLocalizator}). On the other hand, if $\M$ is not left proper, we will not obtain a BF-reflector, but with a bit more effort and under smallness hypotheses on $\M$, we will see in Section \ref{sect_LBLviaCellularity} that a similar left Bousfield localization exists by other means.  


In the rest of the section, we will consider $\M$ to be a $\V$-model category and we will require more properties on $\M$ and $\V$ when needed, and only those which are explicitly mentioned. For example, we will not assume that $\V$ is a cofibrantly generated model category, only that it admits a set of generating cofibrations.  

\begin{defn} Let $\Srect$ be a set of maps in $\M$. 
	\begin{itemize}
		\item An object $\Xrect$ in $\M$ is $\Srect$-\emph{local} if 
		$
		\mathbb{R}\underline{\M}(\Srect,\Xrect) 
		$ are equivalences in $\V$.
		\item A map $\upf$ in $\M$ is an $\Srect$-\emph{equivalence} if $\mathbb{R}\underline{\M}(\upf,\Xrect)$ is an equivalence for any $\Srect$-local object $\Xrect$. The class of $\Srect$-equivalences will be referred as $\Eq_{\Srect}$.
	\end{itemize}
\end{defn}

\begin{rem} It is easy to see that the above notions only depend on the homotopy type of $\Srect$ in $\Ho\M$. That is why we can and we will assume that $\Srect$ consists on core cofibrations in the sequel. One of the nice implications of this choice is: since $\underline{\M}\colon(\Xrect,\Yrect)\mapsto\underline{\M}(\Xrect,\Yrect)$ is a Quillen bifunctor, a fibrant object $\Xrect$ is $\Srect$-local iff $\underline{\M}(\Srect,\Xrect)$ belongs to $\AFib$.
\end{rem}

Using the usual adjunction on the arrow category induced by the $\V$-enrichment of $\M$, 
$$
\underline{\M}(\Srect,\Xrect) \text{ belongs to }\AFib\quad\text{iff}\quad \Cof\;\boxslash\; \underline{\M}(\Srect,\Xrect) \quad\text{iff}\quad \big(\Cof\square\Srect\big)\boxslash \big(\Xrect\to\mathbb{1}\big).
$$
This motivates the definition of the following set of morphisms to detect when an object is $\Srect$-local, provided it is fibrant and that $\V$ has a set of generating cofibrations.
\begin{defn} The set of $\Srect$-\emph{horns} is the following set of maps in $\M$
	$$
	\mathsf{Horn}(\Srect)=\Irect_{\V}\square\Srect=\left\{\upi\square\upf \text{ s.t. }\upi\in\Irect_{\V} \text{ and }\upf\in\Srect\right\}
	$$
where $\Irect_{\V}$ is a fixed set of generating cofibrations of $\V$.
\end{defn}

In general, we also need to detect if an object is fibrant before checking if it is $\Srect$-local, or perhaps detect both properties simultaneously. One way to do so is by adding to $\mathsf{Horn}(\Srect)$ a set of maps $\Jrect$ of $\M$ that detects fibrant objects by rlp against $\Jrect$. Then, considering the set  $\overline{\mathsf{Horn}}(\Srect)=\mathsf{Horn}(\Srect)\cup\Jrect$, it is clear that $\Xrect$ in $\M$ is fibrant $\Srect$-local iff it has right lifting property against $\overline{\mathsf{Horn}}(\Srect)$\footnote{In usual references like \cite{hirschhorn_model_2003}, $\Jrect$ is a set of generating acyclic cofibrations. We use this other approach to stress that this choice is not crucial and since this flexibility has its effect in Section \ref{sect_Applications}.}.

\begin{rem} There are model categories that come with a set that detect fibrant objects, but not a set of generating acyclic cofibrations. Any right Bousfield localization at a set of objects of a right proper cellular or combinatorial model category is an example, see  \cite{barwick_left_2010} or \cite{hirschhorn_model_2003}. 
\end{rem}

If one is able to apply a small object argument to the set $\overline{\mathsf{Horn}}(\Srect)$ to factorize the unique map $\Zrect\to\mathbb{1}$, it will result a map $\Zrect\to\widetilde{\Zrect}$ which belongs to $\overline{\mathsf{Horn}}(\Srect)\text{-cell}$ and where  $\widetilde{\Zrect}$ is fibrant $\Srect$-local. This looks like a BF-reflector and this is exactly the strategy to construct it. What has to be checked to get an actual BF-reflector is that maps in  $\overline{\mathsf{Horn}}(\Srect)\text{-cell}$ are $\Srect$-equivalences. This will not hold in general and here is where left properness of $\M$ plays a role. The following lemmas discuss this problem without assuming that $\M$ is left proper.

\begin{lem}\label{lem_HornsAreCoreCofibrations} If $\V$ admits a generating set of cofibrations with cofibrant domains, the associated class of maps  $\mathsf{Horn}(\Srect)$ in $\M$ consists of core cofibrations.
\end{lem}
\begin{proof}
 This follows from the fact that maps in $\Srect$ are core cofibrations as well as those in $\Irect_{\V}$ by the pushout product axiom. 	
\end{proof}

\begin{lem}\label{lem_HornCellsAreSEquivalences} Assume (i) $\V$ admits a generating set of cofibrations with cofibrant domains and (ii) $\M$ admits a set $\Jrect$ of core acyclic cofibrations that detect fibrant objects by rlp. Then, 
	 any map in $\overline{\mathsf{Horn}}(\Srect)\textup{-cell}$ with cofibrant source is an $\Srect$-equivalence.
\end{lem}
\begin{proof}
We are reduced to show that maps in $\mathsf{Horn}(\Srect)$ are $\Srect$-equivalences and that the steps in the construction of cell complexes preserve this property.

First, we check that $\mathsf{Horn}(\Srect)\subseteq \Eq_{\Srect}$. Since any map $\upi\square\upf$ in $\mathsf{Horn}(\Srect)$ is a core cofibration by Lemma \ref{lem_HornsAreCoreCofibrations}, we show that $\upi\square\upf$ is an  $\Srect$-equivalence by showing that $\underline{\M}(\upi\square\upf,\Xrect)$ is an acyclic fibration in $\V$ for any fibrant $\Srect$-local object $\Xrect$. Using the set of generating cofibrations $\Irect_{\V}$ of $\V$, this fact follows from the following adjunction argument  
$$
\Irect_{\V}\,\boxslash\,\underline{\M}(\upi\square\upf,\Xrect) \quad \text{iff} \quad
 \big(\Irect_{\V}\square\upi\big)\,\boxslash\,\underline{\M}(\upf,\Xrect)
$$
by the pushout product axiom. 

Now, we prove the cobase change case, which has to be seen as attaching cells to a cofibrant object. So, let us take a map $\upj$ which fits into a pushout square
\begin{equation}\label{eqt_PushoutOfHorn}
\begin{tikzcd}[ampersand replacement=\&]
\bullet \ar[rr,"\in\mathsf{Horn}(\Srect)"]\ar[d]\ar[rrd, phantom, "\ulcorner" description] \&\& \bullet\ar[d]\\
\Zrect\ar[rr, "\upj"'] \&\& \Zrect'
\end{tikzcd}
\end{equation}
where $\Zrect$ is cofibrant. By Lemma \ref{lem_HornsAreCoreCofibrations}, this pushout is a homotopy pushout and hence becomes a homotopy pullback after applying $\mathbb{R}\underline{\M}(\star,\Xrect)$. Thus, when $\Xrect$ is $\Srect$-local, $\mathbb{R}\underline{\M}(\upj,\Xrect)$ is an equivalence since $\mathsf{Horn}(\Srect)\subseteq\Eq_{\Srect}$. 

Finally, consider that $\upg$ is the transfinite composite of a $\uplambda$-sequence
\begin{equation}\label{eqt_TransfiniteCompositeOfHorn}
\begin{tikzcd}[ampersand replacement=\&]
\Zrect \ar[r,"\upg_{1}"] \& \bullet \ar[r] \& \cdots \ar[r] \& \bullet \ar[r, "\upg_{\upalpha}"] \& \bullet \ar[r]\& \cdots
\end{tikzcd}
\end{equation}
where each $\upg_{\upalpha}$ is a cobase change of a map in $\overline{\mathsf{Horn}}(\Srect)$ and $\Zrect$ is cofibrant. Again by Lemma \ref{lem_HornsAreCoreCofibrations} and the cobase change case above, an application of $\underline{\M}(\star, \Xrect)$ to the $\uplambda$-sequence describes $\underline{\M}(\upg,\Xrect)\simeq \mathbb{R}\underline{\M}(\upg,\Xrect)$ as a transfinite cocomposite of acyclic fibrations, proving that it is an acyclic fibration.
\end{proof}

\begin{rem}
The cofibrant source assumption in Lemma \ref{lem_HornCellsAreSEquivalences} is essential to ensure that the pushout (\ref{eqt_PushoutOfHorn}) and the transfinite composite (\ref{eqt_TransfiniteCompositeOfHorn}) are homotopical (which is the main ingredient for loc.cit.\;to hold). It is folklore that asking $\M$ to be left proper allows one to drop this cofibrancy hypothesis (see \cite[Chapter 4]{hirschhorn_model_2003}). Moreover, assuming that $\M$ is left proper allows us to replace the set $\mathsf{Horn}(\Srect)$ by core cofibrations and $\Jrect$ by core acyclic cofibrations if all maps in $\Jrect$ are equivalences even when $\Irect_{\V}$ do not have cofibrant domains and $\Jrect$ is arbitrary. Hence, in the left proper case, under an unsignificant change of $\overline{\mathsf{Horn}}(\Srect)$, relative $\overline{\mathsf{Horn}}(\Srect)$-cell complexes are $\Srect$-equivalences.
\end{rem}

We state first what is known without left properness. In this case, we will have to apply Lemma \ref{lem_HornCellsAreSEquivalences} and therefore we obtain:
\begin{prop}\label{prop_LocalizatorFromAlmostNothing} Let $\M$ be a $\V$-model category, $\Srect$ a set of maps in $\M$ where  (i) $\V$ admits a set of generating cofibrations with cofibrant domains and (ii) $\M$ admits a set of core acyclic cofibrations that detect fibrant objects. If the set of core cofibrations  $\overline{\mathsf{Horn}}(\Srect)$ admits Quillen's soa, 
then there is a zigzag of natural transformations over $\M$, 
$
\id\xLeftarrow{\upq} \bullet \xRightarrow{r} \Rfrak,
$ 
such that:
\begin{itemize}
\item $\upq$ is a natural equivalence and $r$ is a natural $\Srect$-equivalence.
\item The endofunctors in the zigzag are $\V$-enriched, homotopical and  $\Rfrak$ takes values in (bi)fibrant $\Srect$-local objects.
\end{itemize}
\end{prop}
\begin{proof}
	We construct the zigzag on an object $\Xrect$ (naturality will be clear). Take a cofibrant replacement $\Xrect\leftarrow \Qrect\Xrect$ and factor $\Qrect\Xrect\to \mathbb{1}$ using Quillen's soa for $\overline{\mathsf{Horn}}(\Srect)$. The conclusion follows by Lemma \ref{lem_HornCellsAreSEquivalences}. 
	
	The $\V$-enrichment of the functors comes from applying enriched Quillen's soa instead of ordinary Quillen's soa \cite[Corollary 13.2.4]{riehl_categorical_2014}.
\end{proof}

Assuming that $\M$ is left proper, the first part of the argument in Proposition \ref{prop_LocalizatorFromAlmostNothing} yields an easy generalization of \cite[Theorem 4.3.3]{hirschhorn_model_2003}. In this case, one gets: 
\begin{prop}\label{prop_LocalizatorLeftProperCase}
	 Let $\M$ be a left proper $\V$-model category, $\Srect$ a set of maps in $\M$ satisfying (i) $\V$ admits a set of generating cofibrations (no hypothesis on domains) and (ii) $\M$ admits a set of maps that detect fibrant objects. If the set of core cofibrations $\overline{\mathsf{Horn}}(\Srect)$ admits Quillen's soa\footnote{Or any generalization which produces factorizations whose left factor is a cellular complex.}, 
	then there is a ($\V$-enriched) BF-reflector $(\Rfrak,\upeta)$ such that:
	\begin{itemize}
		\item  $\Rfrak$-equivalences are precisely $\Srect$-equivalences.
		\item $\Rfrak$-local objects are precisely fibrant $\Srect$-local objects.
	\end{itemize}
\end{prop}

\begin{rems}\label{rems_FindingLocalizatorSEMIMODELS} $\,$
	\begin{enumerate} 
		\item In this section, one can replace ``$\V$-model" by ``$\V$-semimodel" category, where being $\V$-enriched means as in Section \ref{sect_Preliminaries}. However, $\V$ must be closed symmetric monoidal model category in either case and, as a technicality to obtain $\V$-enriched functors, it has to be assumed that for any $\Vrect\in\V$, the tensoring $\Vrect\odot\,\star$ is core left Quillen. Note that the zigzag in Proposition \ref{prop_LocalizatorFromAlmostNothing} is longer if we start from a right $\V$-semimodel category.
		
		\item Using the theory of frames, one may adapt the results for non-enriched settings. This is precisely the perspective chosen by Hirschhorn in \cite{hirschhorn_model_2003}. 
	\end{enumerate}
\end{rems}

\begin{thm}\label{thm_LBLatSequivalencesWith1StepLocalizator}
 Let $\M$ be a left proper (Fresse/coSpitzweck right semi)model category that admits a set of maps $\Jrect$ that detects fibrant objects and let $\Srect$ be a set of maps in $\M$. If the set of core cofibrations $\overline{\mathsf{Horn}}(\Srect)$ admits Quillen's soa, then   
 $$
 \M_{\Srect}=(\Mrect,\Cof,\Eq_{\Srect},(\Cof\cap\Eq_{\Srect})^{\boxslash})
 $$
 is the left Bousfield localization of $\M$ at $\Srect$-equivalences among  Fresse/coSpitzweck right semimodel categories with left Quillen functors between them.	
 
 In the enriched case, i.e.\ if $\M$ is a Fresse/coSpitzweck right $\V$-semimodel category, the $\V$-enriched left Bousfield localization of $\M$ at $\Srect$ exists as a Fresse/coSpitzweck right $\V$-semimodel category provided (a) $\V$ admits a set of generating cofibrations and (b) the set of core cofibrations  $\overline{\mathsf{Horn}}(\Srect)$ admits Quillen's soa.
\end{thm}
\begin{proof}
	By Proposition \ref{prop_LocalizatorLeftProperCase} and Remarks \ref{rems_FindingLocalizatorSEMIMODELS}, there is a ($\V$-enriched) BF-reflector to which we can apply Theorem \ref{thm_LBLfromLocalizator} (and Proposition \ref{prop_EnrichmentForLocalizator} in the enriched case). The definition of $\mathsf{Horn}(\Srect)$ in the non-enriched setting is analogous to \cite[Definition 4.2.1]{hirschhorn_model_2003}.
\end{proof}

\begin{rem} Note that Theorem \ref{thm_LBLatSequivalencesWith1StepLocalizator}  can be applied to non-cofibrantly generated model structures. For instance,  right Bousfield localizations constructed via \cite[Theorem 5.1.1]{hirschhorn_model_2003} are not known to be cofibrantly generated but they are susceptible to fit into Theorem  \ref{thm_LBLatSequivalencesWith1StepLocalizator} if one shows that the particular example at sight is left proper (e.g.\ any stable right Bousfield localization of a proper cellular or combinatorial stable model category due to \cite[Proposition 4.8]{barnes_stable_2014}). One can also study left proper model structures on presentable categories (e.g.\ model categories with accessible algebraic wfs cofibrantly generated by a small category, but not by a set, with all objects fibrant; see \cite{bourke_algebraic_2016}) or Hurewicz type model structures on topologically bicomplete categories as some of our examples in Section \ref{sect_Applications} suggest. In that section, we also provide a source of examples coming from relative homological algebra (see Theorem \ref{thm:LBLOfRelativeHomologicalAlgebra}).
\end{rem}


\section{Left Bousfield localization \`a la Hirschhorn}\label{sect_LBLviaCellularity}
	
In this section we construct the left Bousfield localization of $\M$ at $\Srect$-equivalences using Proposition  \ref{prop_LocalizatorFromAlmostNothing}. We are going to see that such localization exists as a Spitzweck left semimodel category\footnote{Notice the strong contrast with Theorems \ref{thm_LBLfromLocalizator} and \ref{thm_LBLatSequivalencesWith1StepLocalizator}, where left Bousfield localization is shown to exists as a right semimodel category.} under somehow strong smallness conditions on $\M$ (see Theorem \ref{thm_BousfieldLocalizationCellularCase}). 

For such goal, we revisit the material in \cite{hirschhorn_model_2003} about localization of cellular model structures but working in an enriched setting and  incorporating a fundamental remark of Goerss-Hopkins \cite{goerss_moduli_2005}. A thorough source where a particular case was discussed is \cite{harper_topological_2019}, although only for operadic algebras and for a concrete choice of ``homological" equivalences.

 From now on, in this section, we will consider that $\M$ is a cofibrantly generated Fresse/ Spitzweck left $\V$-semimodel category and $\Srect$ a set of maps in $\M$ (which are assumed to be core cofibrations without loss of generality). Also, $\V$ is a closed symmetric monoidal model category admitting a set of generating cofibrations with cofibrant domains in the remainder of the section. Further conditions on $\M$ will be imposed when they become necessary for the arguments to work, in order to emphatise their dependence. 

\begin{paragraph}{Towards a left semimodel structure.}
		    Let us start by explaining a wrong approach to the construction of a left $\V$-semimodel structure on the pair  $(\Mrect,\Eq_{\Srect})$, but which hopefully illuminates the line of arguments.

    Our goal is to check if 
    $
    \Lrect_{\Srect}\M=(\Mrect,\Cof,\Eq_{\Srect},(\Cof\cap\Eq_{\Srect})^{\boxslash})
    $ 
    is a (Fresse/Spitzweck left $\V$-semi)model structure. In the best possible scenario, from $\M$ and a set of horns $\mathsf{Horn}(\Srect)$ associated to $\Srect$, one would obtain the desired model structure. A natural way to verify this claim would be appealing to Kan's recognition principle 
     \cite[Theorem 11.3.1]{hirschhorn_model_2003} for  $\overline{\mathsf{Horn}}(\Srect)$ and the already present weak factorization system $(\Cof,\AFib)$ in $\M$:
	\begin{enumerate}
		\item[(i)] (Compatibility of acyclic cofibrations)
		$
		\overline{\mathsf{Horn}}(\Srect)\text{-cof}\subseteq\Cof\cap\Eq_{\Srect}.
		$ 
		\item[(ii)](Compatibility of acyclic fibrations)
		$
		\overline{\mathsf{Horn}}(\Srect)^{\boxslash}\cap\Eq_{\Srect}\supseteq\AFib.
		$
		\item[(iii)](Full compatibility) The reverse inclusion holds in (i) or (ii).
	\end{enumerate}
    However, as noted by Bousfield \cite[Example 2.1.6]{hirschhorn_model_2003}, this strategy will not work in general by the violation of (iii);\footnote{Note that (i) may not hold neither. That is why we use an extra cofibrancy assumption in Lemma \ref{lem_HornCellsAreSEquivalences} or $\M$ is assumed to be left proper in \cite[Chapter 4]{hirschhorn_model_2003}.} indeed, Bousfield's example shows that  $\overline{\mathsf{Horn}}(\Srect)\text{-cof}_{\circ}\subsetneq\Cof_{\circ}\cap\Eq_{\Srect}.$ 
    The problem is that there may be more acyclic cofibrations in the localized model that those constructed from $\overline{\mathsf{Horn}}(\Srect)$.\footnote{In fact, this problem also holds for core acyclic cofibrations, forbidding the use of $\overline{\mathsf{Horn}}(\Srect)$ to construct a left semimodel structure.} 
    
    To improve this strategy, we must just find a set $\Jrect_{\Srect}$ which replaces $\overline{\mathsf{Horn}}(\Srect)$ to generate $\Cof\cap\Eq_{\Srect}$. This is a highly non-trivial task. At present, $\Jrect_{\Srect}$ is only known to exist under cardinality/boundedness assumptions on $\M$ and no workable description of $\Jrect_{\Srect}$ is available in general. These cardinality assumptions can be codified by restricting $\M$ to be a cellular or a combinatorial (semi)model category. Since the combinatorial case is explained in \cite{white_left_2020}, we will treat the cellular one. 
    
    The proofs presented here are just adaptations of those in \cite{hirschhorn_model_2003} to left semimodel categories, and so the presentation will be mostly expository. In contrast with the main sources, we will not assume that $\M$ is cellular when it is not mandatory, to remark that there are just a few points where it is used that cofibrations are effective monomorphisms.
    
    The first step to find the set $\Jrect_{\Srect}$ of generating acyclic cofibrations is to reduce the class of morphisms that detect the canonical candidates for fibrations in the localized model. Start with the whole class $\Cof\cap\Eq_{\Srect}$. The reductions are considered by checking that it suffices to take each time a smaller class in $\Cof\cap\Eq_{\Srect}$ which satisfy (ii) and its reverse for (iii) in Kan's recognition principle. 
    
    One starts by showing that one can consider core cofibrations instead of plain cofibrations assuming that $\M$ is tractable (see Remark \ref{rems_SemimodelsAndTractability}). The same conclusion holds if $\M$ is left proper instead of tractable by \cite[Proposition 13.2.1]{hirschhorn_model_2003}.
    \begin{prop}\label{prop_RestrictGeneratingACOFtoCore} 
    	If $\M$ is tractable, the following classes of maps are equal
    	$$
    	\AFib=(\Cof\cap\Eq_{\Srect})^{\boxslash}\cap\Eq_{\Srect}=(\Cof_{\circ}\cap\Eq_{\Srect})^{\boxslash}\cap\Eq_{\Srect}.
    	$$
    \end{prop}
\begin{proof}
See	\cite[Proposition 3.8]{harper_topological_2019}.
\end{proof}

    Next, one replaces core cofibrations by cellular inclusions of subcomplexes.
    \begin{prop}\label{prop_RestrictGeneratingACOFtoSubcells} Let $f$ be a fibration in $\M$. Then $f$ has the right lifting property against $\Cof_{\circ}\cap \Eq_{\Srect}$ iff it has so with respect to $\Irect_{\M}\textup{-subcell}\cap\Eq_{\Srect}$. Hence, if $\M$ is tractable or left proper, 
    	$$\AFib=\big(\Jrect_{\M}\cup\,(\Irect_{\M}\textup{-subcell}\cap\Eq_{\Srect})\big)^{\boxslash}\cap \Eq_{\Srect}$$
    where $\Jrect_{\M}$ is a set of generating acyclic cofibrations for $\M$.	
    \end{prop}
\begin{proof}
	See 
	\cite[Proposition 3.8]{harper_topological_2019} or \cite[Proposition 1.5.8]{goerss_moduli_2005}.
\end{proof}

    At this point is where the subtleties come and where the harder axiom for cellularity plays a role. In order to further reduce $\Irect_{\M}\textup{-subcell}\cap\Eq_{\Srect}$ to a set, the strategy is to find a sufficiently large cardinal $\upgamma$ for which it suffices to take $\Irect_{\M}\textup{-subcell}$ inclusions whose target has less than $\upgamma$ cells. We will denote this class of maps by $\Irect_{\M}\textup{-subcell}_{<\upgamma}$ in the sequel.

\end{paragraph}

\begin{paragraph}{Bousfield-Hirschhorn-Smith cardinality argument.}
	The fundamental reduction to find a set of generators for $\Cof\cap\Eq_{\Srect}$ is given by the following result.
\begin{prop}{\cite[Proposition 4.5.6]{hirschhorn_model_2003}}\label{prop_RestrictGeneratingACOFtoBoundedSubcells} Assume that $\M$ is cellular. Then, there exists a cardinal $\upgamma$\footnote{The choice of the cardinal $\upgamma$ is explained below, in Remark \ref{rem_RegularCardinalForBousfieldSmithArgument}.} for which any fibration $f$ in $\M$ satisfies that $f$ belongs to $\big(\Irect_{\M}\textup{-subcell}\cap \Eq_{\Srect}\big)^{\boxslash}$ iff it belongs to  $\big(\Irect_{\M}\textup{-subcell}_{<\upgamma}\cap\Eq_{\Srect}\big)^{\boxslash}$. Thus, if $\M$ is tractable and cellular 
    	$$\AFib=\underset{\Jrect_{\Srect}}{\underbrace{\big(\Jrect_{\M}\cup\,(\Irect_{\M}\textup{-subcell}_{<\upgamma}\cap\Eq_{\Srect})\big)}}^{\boxslash}\cap \Eq_{\Srect}$$
    	where $\Jrect_{\M}$ is a set of generating acyclic cofibrations for $\M$.
    \end{prop}
    \begin{proof}[Sketch proof.] The claim clearly reduces to check if the bounded lifting property suffices to obtain lifts for all cellular inclusions of subcomplexes. Thus, for $\upg\in\Irect_{\M}\text{-subcell}\cap\Eq_{\Srect}$, one applies Zorn's Lemma to the poset of partial solutions to a given lifting problem $\upg\Rightarrow\upf$ (solutions to the lifting problem defined only over a subcomplex) to find a global solution. One argues by contradiction assuming that the maximal element of the poset is not a full solution to $\upg\Rightarrow \upf$.
    The contradiction arises from the fact that one can extend such maximal element to a strictly bigger partial solution by adding a controlled number of cells using Lemma \ref{lem_ConstructionOfBiggerSEquivalentSubcomplex}.
    \end{proof}

    The next lemma, alluded in the sketch proof of Proposition \ref{prop_RestrictGeneratingACOFtoBoundedSubcells}, is just an easy variation of \cite[Proposition 4.5.5]{hirschhorn_model_2003} for left semimodel categories. 
    \begin{lem}\label{lem_ConstructionOfBiggerSEquivalentSubcomplex}
    	Assume that $\M$ is cellular. Then, there exists a cardinal $\upgamma$ for which the following holds: if $\Xrect'\subset \Yrect$ is a proper cellular inclusion of $\Irect_{\M}$-cellular complexes which is also an $\Srect$-equivalence, there is a subcomplex $\Zrect$ of $\Yrect$ such that: (1) $\Zrect$ has less than $\upgamma$ cells; (2) $\Zrect$ is not included in $\Xrect'$; and (3) the following 
    		 square of cellular inclusions is bicartesian (both a pullback and a pushout) and its vertical maps are $\Srect$-equivalences
    		$$
    		\begin{tikzcd}[ampersand replacement=\&]
    		\Zrect\cap \Xrect' \ar[r]\ar[d]\ar[rd,phantom,"\square" description]\& \Xrect'\ar[d]\\
    		\Zrect\ar[r] \& \Zrect\cup \Xrect'
    		\end{tikzcd}.
    		$$
    \end{lem}
    \begin{proof}[Sketch proof.]
    	Note that left properness is not needed anywhere in \cite[Proposition 4.5.5]{hirschhorn_model_2003}. We provide an outline of this proof for completeness.
    	
    	The subcomplex $\Zrect$ is built via a small object type construction: it starts at a subcomplex $\Zrect_{0}$ satisfying $(1)$ and $(2)$, which exists since $\Xrect'\subset \Yrect$ is a proper inclusion, and it transfinitely adds cells to make the vertical maps in $(3)$ closer to be $\Srect$-equivalences. 
    	
    	What has to be found is a way to add a controlled number of cells in each step to get closer to condition $(3)$. The upshot is to apply a functor $\Rfrak$ fulfilling the conclusions of Proposition \ref{prop_LocalizatorFromAlmostNothing} to the bicartesian square (since it detects $\Srect$-equivalences between cofibrant objects). Then, one crucially notes that the resulting square consists of cellular inclusions between fibrant cellular complexes. This allows one to control with a cardinal $\upgamma$ how much cells are needed to get closer to make $\Rfrak(\Zrect\cap \Xrect')\to \Rfrak(\Zrect) $ and $\Rfrak(\Xrect')\to \Rfrak(\Zrect\cup \Xrect')$ equivalences (in fact, deformation retracts) in each step of the transfinite construction. 	
    \end{proof}

    \begin{rem}\label{rem_RegularCardinalForBousfieldSmithArgument}
    	Lemma \ref{lem_ConstructionOfBiggerSEquivalentSubcomplex} requires a fine analysis of the interplay between a functor which detects $\Srect$-equivalences and the category of $\Irect_{\M}$-cell complexes and cellular inclusions between them. In the setting treated in this work, such functor $\Rfrak$ fulfils the conclusions of Proposition \ref{prop_LocalizatorFromAlmostNothing} (indeed, a variant of its construction which restricts to the category of cellular complexes and cellular inclusions as in \cite[Section 4.4]{hirschhorn_model_2003}). In different settings, for instance the ones discussed in \cite{goerss_moduli_2005, harper_topological_2019}, homology type functors are applied to detect the new class of equivalences. 
    	
    	The mentioned analysis is what leads us to set the cardinal $\upgamma$. It is chosen so that the following conditions, applied in \cite[Proposition 4.5.5]{hirschhorn_model_2003}, hold:
    	\begin{itemize}
    		\item Cell complexes with $<\upgamma$ cells are $\upgamma$-compact,
    		\item For any cell of a cell complex, there is a subcomplex which contains it with $<\upgamma$ cells,
    		\item $\Rfrak$ preserves cell complexes with $<\upgamma$ cells,
    		\item There is a cylinder construction which preserves cell complexes with $<\upgamma$ cells. 
    	\end{itemize}
    	Formally, the cardinal $\upgamma$ is specified in \cite[Definition 4.5.3]{hirschhorn_model_2003}. 
    \end{rem}
    
    \begin{rem}\label{rem_EffectiveMonoForIntersections}
    	The proof of Lemma \ref{lem_ConstructionOfBiggerSEquivalentSubcomplex} and hence of Proposition \ref{prop_RestrictGeneratingACOFtoBoundedSubcells} are the only places where the full strength of cellularity is applied. More concretely, that cellular maps are effective monomorphisms is considered just to ensure existence of intersections of subcomplexes and control over their cells in Lemma \ref{lem_ConstructionOfBiggerSEquivalentSubcomplex} (see \cite[Diagram (1.5.1) and Lemma 1.5.3]{goerss_moduli_2005}). Therefore, when working with cellular model categories, one can relax the effective monomorphism condition to hold just for core cofibrations (or cellular inclusions between cell complexes).
    \end{rem}
\end{paragraph}

\begin{paragraph}{Left Bousfield localization of left semimodel structures.}
		   We are ready to state the main result of this section. Recall that $\V$ is a closed symmetric monoidal model category admitting a set of generating cofibrations with cofibrant domains.

	\begin{thm}\label{thm_BousfieldLocalizationCellularCase} Let $\M$ be a tractable cellular Fresse (resp. Spitzweck) left $\V$-semimodel and $\Srect$ a set of maps in $\M$. Then, the left Bousfield localization of $\M$ at $\Srect$ exists as a tractable cellular Fresse (resp. Spitzweck) left $\V$-semimodel category and its bifibrant objects are those bifibrant objects of $\M$ which are $\Srect$-local. 
	\end{thm}
	\begin{proof} We prove that
		$
	    \Lrect_{\Srect}\M=\big(\Mrect,\Cof,\Eq_{\Srect},(\Cof\cap\Eq_{\Srect})^{\boxslash}\big)
		$
		is a Spitzweck left semimodel category by applying a variation of Kan's recognition principle for these structures (Fresse's case follows the same arguments). It suffices to check that the sets $\Irect_{\M}$ (generating cofibrations) and $\Jrect_{\Srect}$ (generating acyclic cofibrations), defined in  Proposition \ref{prop_RestrictGeneratingACOFtoBoundedSubcells},  satisfy 
		\begin{itemize}
			\item[(i)](Compatibility of core acyclic cofibrations) $\Jrect_{\Srect}\text{-cof}_{\circ}\subseteq \Irect_{\M}\text{-cof}_{\circ}\cap \Eq_{\Srect}$,
			\item[(ii)](Compatibility of acyclic fibrations) ${\Jrect_{\Srect}}^{\boxslash}\cap\Eq_{\Srect}={\Irect_{\M}}^{\boxslash}=\AFib$.
		\end{itemize}
		By definition of $\Jrect_{\Srect}$, (i) is analogous to Lemma \ref{lem_HornCellsAreSEquivalences} and (ii) is the content of Proposition \ref{prop_RestrictGeneratingACOFtoBoundedSubcells}.

        The resulting structure $\Lrect_{\Srect}\M$ is a $\V$-semimodel category. This claim follows if we prove that the $\V$-tensor is a core left Quillen bifunctor by tractability of $\V$ and $\M$ plus an obvious variation of Proposition \ref{prop_WeakQuillenPairForCoreCofibGen}. Thus, as in the proof of Proposition \ref{prop_EnrichmentForLocalizator}, we are reduced to show that $\upi\square \upj$ is an $\Srect$-equivalence provided that $\upi$ and $\upj$ are core cofibrations and $\upj$ is also an $\Srect$-equivalence. By definition, we have to prove that $\underline{\M}(\upi\square\upj,\Zrect)$ is a core acyclic fibration for any fibrant $\Srect$-local object $\Zrect$. As usual, this fact follows from inspection over 
        $$
        \Irect_{\V}\,\boxslash\,\underline{\M}(\upi\square\upj,\Zrect) \quad \text{iff} \quad
        \big(\Irect_{\V}\square\upi\big)\,\boxslash\,\underline{\M}(\upj,\Zrect).
        $$
 		
		The universal property of $\Lrect_{\Srect}\M$ can be proved just observing that the identity functor  $\id\colon\M\to\Lrect_{\Srect}\M$ is a core left Quillen functor. This implies by tractability and Proposition \ref{prop_WeakQuillenPairForCoreCofibGen} that $\id\colon\M\rightleftarrows\Lrect_{\Srect}\M\colon \id$ is a Quillen $\V$-adjunction of left semimodel categories. 
		
		Let us finish by characterizing bifibrant objects of $\Lrect_{\Srect}\M$. These are $\Srect$-local since this condition can be encoded as having right lifting property against $\overline{\mathsf{Horn}}(\Srect)$ and this set belongs to $\Cof_{ \circ}\cap\Eq_{\Srect}$ by tractability and Lemmas \ref{lem_HornsAreCoreCofibrations}, \ref{lem_HornCellsAreSEquivalences}. For the reverse implication, take a map  $\upj\in\Cof_{ \circ}\cap\Eq_{\Srect}$, a bifibrant $\Srect$-local object $\Zrect$ and a lifting problem $\upj\Rightarrow(\Zrect\to\mathbb{1})$. By definition of $\Srect$-equivalence, we have a bijection of hom-sets in the homotopy category of $\M$,  
		$$
		\upj^*\colon\Ho\M(\source(\upj),\Zrect)\overset{\simeq}{\longrightarrow}\Ho\M(\target(\upj),\Zrect).
		$$
		This bijection combined with the fact that $\upj$ is a core cofibration and $\Zrect$ bifibrant in $\M$ yields a solution to the lifting problem up to (right) homotopy. Use the homotopy extension property of $\upj$ to obtain an actual solution for $\upj\Rightarrow(\Zrect\to\mathbb{1})$ from that up to homotopy solution.
	\end{proof}
	
A non enriched version of Theorem \ref{thm_BousfieldLocalizationCellularCase} also holds by similar arguments.
	
\begin{rem} We were not able to provide a characterization of fibrant objects in Theorem \ref{thm_BousfieldLocalizationCellularCase}. The issue comes when one tries to show that fibrant $\Srect$-local implies fibrant in $\Lrect_{\Srect}\M$. Replicating the idea in the cited result to characterize bifibrant objects, the best one can get is an up to left homotopy solution to any lifting problem of a map in $\Cof_{\circ}\cap\Eq_{\Srect}$ against a fibrant $\Srect$-local object. We do not know how to find an actual solution from this up to homotopy solution without assuming that fibrant $\Srect$-local objects admit path objects.
\end{rem}

However, following an idea of D.White, one can obtain a characterization of fibrant objects in the non-enriched case and under a weak hypothesis in the enriched case.

\begin{prop}\label{prop:CharactOfFibrantObjectsInHirschhornLBL}
Under the hypotheses of Theorem \ref{thm_BousfieldLocalizationCellularCase}, the fibrant objects of $\Lrect_{\Srect}\M$ are the $\Srect$-local fibrant objects of $\M$ provided the monoidal unit $\mathbb{I}_{\V}\in\V$ is cofibrant.

In the non-enriched version of Theorem \ref{thm_BousfieldLocalizationCellularCase}, the claim always holds. 
\end{prop}
\begin{proof} Let us explain the $\V$-enriched case. As explained in the cited theorem and following remark, it remains to see that every $\Srect$-local and fibrant object $\Zrect\in\M$ has rlp against any map $\upj\in\Cof_{ \circ}\cap\Eq_{\Srect}$. Since the bifunctor $\underline{\M}\colon\M^{\text{op}}\times\M\to \V$ is core left Quillen and $\Zrect$ is  $\Srect$-local, $\underline{\M}(\upj,\Zrect)$ is a core acyclic fibration in $\V$. Then, $\mathbb{I}_{\V}$ being cofibrant has llp against $\underline{\M}(\upj,\Zrect)$. One concludes the claim by an usual adjunction argument:
	$$
	\mathbb{I}_{\V}\,\boxslash\,\underline{\M}(\upj,\Zrect) \quad \text{iff} \quad 
	\big(\mathbb{I}_{\V}\otimes\upj\big)\,\boxslash\,\Zrect \quad \text{iff} \quad \upj\,\boxslash\,\Zrect.
	$$
	
The non-enriched case follows from the analogous analysis with homotopy function complexes (\cite[Proposition 17.8.9]{hirschhorn_model_2003} adapted to left semimodel categories). The proof of such proposition and the results that it requires hold in this situation because we just use core cofibrations.
\end{proof}
\end{paragraph}

\section{Applications}\label{sect_Applications}
	Before discussing applications, we summarize how the theory of Bousfield localization and semimodel structures interact in the following table. Note that there is an analogous table obtained by replacing all instances of ``Spitzweck" and ``coSpitzweck" below by ``Fresse".
		$$
		\begin{array}{|l|l|cc|c|l|}
		\hline 
		\multicolumn{2}{|c|}{\textit{Bousfield localization}} & \multicolumn{2}{c|}{\textit{Input/Output structure}} & \multicolumn{2}{c|}{\textit{Existence result}} \\ \hline \hline
		\multicolumn{2}{|c|}{
			\parbox[c][15mm]{0mm}{}
		\begin{array}{c}
		\text{\small{left Bousfield localization}}\\[-.5mm]
		\text{\small{via BF-reflector}}
		\end{array}
		}                                  &   \multicolumn{2}{c|}{\parbox[c][15mm]{0mm}{}
			\begin{array}{c}
			\text{\small{coSpitzweck right}}\\[-.5mm]
			\text{\small{semimodel category}}
			\end{array} }                       & \multicolumn{2}{c|}{\text{\small{Theorem }} \ref{thm_LBLfromLocalizator}}                        \\ \hline
		\multicolumn{2}{|c|}{
			\parbox[c][15mm]{0mm}{}
		\begin{array}{c}
		\text{\small{left Bousfield localization in}}\\[-.5mm]
		\text{\small{cellular/combinatorial settings}}
		\end{array}
	}                                  &   \multicolumn{2}{c|}{\parbox[c][15mm]{0mm}{}
		\begin{array}{c}
		\text{\small{Spitzweck left}}\\[-.5mm]
		\text{\small{semimodel category}}
		\end{array} }                       & \multicolumn{2}{c|}{
		\parbox[c][15mm]{0mm}{}
			\begin{array}{c}
			\text{\small{Theorem }}\ref{thm_BousfieldLocalizationCellularCase} \text{\small{ and}}\\[-.5mm]
			\cite[\text{\small{Theorem A}}]{white_left_2020}
			\end{array}		
}                     \\ \hline
     \multicolumn{2}{|c|}{
     	\parbox[c][15mm]{0mm}{}
     	\begin{array}{c}
     	\text{\small{right Bousfield localization}}\\[-.5mm]
     	\text{\small{via BF-coreflector}}
     	\end{array}
     }                                  &   \multicolumn{2}{c|}{\parbox[c][15mm]{0mm}{}
     	\begin{array}{c}
     	\text{\small{coSpitzweck left}}\\[-.5mm]
     	\text{\small{semimodel category}}
     	\end{array} }                       & \multicolumn{2}{c|}{\text{\small{dual of Theorem }} \ref{thm_LBLfromLocalizator}}                        \\ \hline
     \multicolumn{2}{|c|}{
     	\parbox[c][15mm]{0mm}{}
     	\begin{array}{c}
     	\text{\small{right Bousfield localization in}}\\[-.5mm]
     	\text{\small{tractable settings}}
     	\end{array}
     }                                  &   \multicolumn{2}{c|}{\parbox[c][15mm]{0mm}{}
     	\begin{array}{c}
     	\text{\small{Spitzweck right}}\\[-.5mm]
     	\text{\small{semimodel category}}
     	\end{array} }                       & \multicolumn{2}{c|}{\cite[\text{\small{Theorem 5.19}}]{barwick_left_2010}}                      \\ \hline
	\end{array}
	$$

In the common ground of left Bousfield localization of a model category $\M$ at a set of maps $\Srect$, the differences between the existence results above become transparent. The left Bousfield localization of $\M$ at $\Srect$ exists among
\begin{itemize}
	\item coSpitzweck right semimodel categories if $\M$ is left proper (assuming that there is a set that detects fibrant objects by rlp).
	\item Spitzweck left semimodel categories if $\M$ is cellular or combinatorial for generating sets of (acyclic) cofibrations with cofibrant domain (tractability).
\end{itemize}
Concerning right Bousfield localization at a set of objects, we have not found a suitable construction of BF-coreflector to apply the dual of Theorem \ref{thm_LBLfromLocalizator} dualizing the results in Section \ref{sect_LocalizatorsFromSmallData}. That is why the comparison with Barwick's construction in \cite[Section 5]{barwick_left_2010} is not given.

We next show some applications of Theorems \ref{thm_LBLfromLocalizator} and \ref{thm_BousfieldLocalizationCellularCase}. 

\paragraph{Localizations of Hurewicz and mixed model structures.} 

Let $\mathsf{Top}$ be the category of compactly generated weak Hausdorff (cgwh) spaces. It carries three closed monoidal model structures as discussed in \cite[Chapter 17]{may_more_2012}: (1) the Hurewicz model structure $\mathsf{Top}_{\text{h}}$, (2) the Quillen model structure $\mathsf{Top}_{\text{q}}$, (3) the mixed model structure $\mathsf{Top}_{\text{m}}$.\footnote{Some brief comments. The Hurewicz model structure is also called Str\o m model structure. The existence of mixed model structures is discussed in \cite[Section 17.3]{may_more_2012} together with their properties. As pointed out by D.White, May-Ponto work with compactly generated (cg) spaces in \cite{may_more_2012} and we have chosen cgwh spaces instead. This difference creates no problems since both choices yield nice closed symmetric monoidal categories of spaces compatible with Hurewicz and Quillen model structures on them.}

It is well known that the left Bousfield localization at any set of maps $\Srect$ of $\mathsf{Top}_{\text{q}}$ exists, but the same question remains open for $\mathsf{Top}_{\text{h}}$ and $\mathsf{Top}_{\text{m}}$ (see \cite[Remark 17.5.7]{may_more_2012}). Our insight is that if such  localizations of $\mathsf{Top}_{\text{h}}$ or $\mathsf{Top}_{\text{m}}$ exist as model categories in general may be a hard question
, but showing that they exist as semimodel categories follows by the results in this work. In fact, we address this question for both model structures in two ways which depend on the chosen mapping space which is going to detect $\Srect$-local objects (see Section \ref{sect_LocalizatorsFromSmallData}).

\begin{rem}
	The existence of left Bousfield localizations for both $\mathsf{Top}_{\text{h}}$ and $\mathsf{Top}_{\text{m}}$ cannot be addressed by previous results since, for example, these model structures are not cofibrantly generated. In \cite{barthel_construction_2013}, it is discussed how one can use Garner's soa to see that $(\ACof,\Fib)$ in $\mathsf{Top}_{\text{h}}$ is a(n algebraic) wfs. This wfs is not generated by a set of maps  and it is shared by both model structures, $\mathsf{Top}_{\text{h}}$ and $\mathsf{Top}_{\text{m}}$.  
\end{rem}

Let us consider $\mathsf{Top}_{\text{h}}$ to lead our argumentation. As a cartesian closed model category, it is equipped with a Quillen bifunctor (the usual mapping space)
$
\Map\colon {\mathsf{Top}_{\text{h}}}^{\text{op}}\times \mathsf{Top}_{\text{h}}\rightarrow \mathsf{Top}_{\text{h}}.
$
Note that $\mathsf{Top}_{\text{h}}$ also carries a simplicial structure, whose associated mapping simplicial set $\Map_{\bullet}$ is simply given by postcomposing $\Map$ with the singular complex functor. Hence, from the set of maps $\Srect$, we obtain two different notions\footnote{For the mixed model structure, the analogous notions coincide since $\mathsf{Top}_{\text{m}}$ is Quillen equivalent to $\mathsf{Top}_{\text{q}}$.} of being $\Srect$-local for a space depending on the usage of $\Map$ or $\Map_{\bullet}$ in the definition. We introduce a bit of notation to note their difference:  
\begin{itemize}
	\item
	a space $\Xrect$ is $(\Srect\vert\,\text{h})$\emph{-local} if $\Map(\upf,\Xrect)$ is an homotopy equivalence for any $\upf\in\Srect$.
	\item
	 a space $\Xrect$ is  $\Srect$\emph{-local} if $\Map(\upf,\Xrect)$ is a weak homotopy equivalence for any $\upf\in\Srect$.
\end{itemize} 
 Of course, $(\Srect\vert\,\text{h})$-\emph{equivalences} (resp. $\Srect$-\emph{equivalences}) are defined in the same way.

\begin{prop}\label{prop_LocalizationsOfHurewiczMixedModelsSIMPLICIALCASE} Let $\Srect$ be a set of continuous maps in $\mathsf{Top}$. Then, the left Bousfield localization of $\mathsf{Top}_{\textup{h}}$ (resp.\;$\mathsf{Top}_{\textup{m}}$) at $\Srect$-equivalences exists among coSpitzweck right semimodel categories. Its fibrant objects are the $\Srect$-local spaces.
\end{prop}
\begin{proof} The proof is the same for both models, so we consider $\mathsf{Top}_{\text{h}}$. Since we are using the simplicial structure of $\mathsf{Top}_{\text{h}}$, we consider the following set of horns
$$
\mathsf{Horn}(\Srect)=\left\{\partial\Delta^{\upn}\hookrightarrow \Delta^{\upn}\text{ with }\upn\geq 0\right\}\square \Srect\footnote{First one has to replace $\Srect$ by (strong) Hurewicz cofibrations, but this is customary so we do not mention it.}.
$$
Because any object in $\mathsf{Top}_{\text{h}}$ is fibrant, $\Zrect\in\mathsf{Top}_{\text{h}}$ is $\Srect$-local iff it satisfies rlp against $\mathsf{Horn}(\Srect)$. Any set of maps in $\mathsf{Top}$ admits Garner's soa (\cite[Theorem 4.4]{garner_understanding_2009}), in particular $\mathsf{Horn}(\Srect)$.  It seems that we can apply Proposition \ref{prop_LocalizatorLeftProperCase} and its consequence, Theorem \ref{thm_LBLatSequivalencesWith1StepLocalizator}, and conclude that the right semimodel category 
$
\mathsf{Top}_{\text{h,}\Srect}=(\mathsf{Top},\Cof_{\text{h}},\Eq_{\Srect},(\Cof_{\text{h}}\cap\Eq_{\Srect})^{\boxslash})
$ 
is the left Bousfield localization of $\mathsf{Top}_{\text{h}}$ at $\Srect$-equivalences, but there is a subtlety. 
 We apply Garner's soa, so the analysis of relative $\mathsf{Horn}(\Srect)$-cell complexes may not be enough. Luckily, (strong) Hurewicz cofibrations are monomorphisms, so  Garner's soa applied to $\mathsf{Horn}(\Srect)$ admits a simplification (\cite[Remark 12.5.6]{riehl_categorical_2014}) for which the left factor in any factorization is a relative $\mathsf{Horn}(\Srect)$-cell complex.
\end{proof}

\begin{prop}\label{prop_LocalizationsOfHurewiczMixedModelsTOPOLOGICALCASE} Let $\Srect$ be a set of continuous maps in $\mathsf{Top}$. Then, the left Bousfield localization of $\mathsf{Top}_{\textup{h}}$ at $(\Srect\vert\,\text{h})$-equivalences exists among coSpitzweck right semimodel categories. Its fibrant objects are the $(\Srect\vert\,\text{h})$-local spaces.
\end{prop}
\begin{proof}
Defining the map $\upf$ as the coproduct of all maps in $\Srect$, the claim follows by an application of Theorem \ref{thm_LBLfromLocalizator} to the strong localization functor $\Lrect^{\text{s}}_{\upf}\colon \mathsf{Top}_{\textup{h}}\to  \mathsf{Top}_{\textup{h}}$ constructed by V.\;Halperin in his PhD thesis \cite{halperin_localization}. 
\end{proof}

 \begin{rem} Halperin's thesis \cite{halperin_localization} is not publicly avalaible, but the construction that we apply above was announced in \cite[proof of Theorem 1.B.5]{farjoun_cellular_1996}. The actual details are involved and we are not going to replicate them here, but at least we note that the construction is applicable in more general settings, as diagrams of topological spaces.\footnote{Halperin uses in \cite{halperin_localization} the category of compactly generated Hausdorff spaces. However, his arguments are completely general and applicable to any convenient category of topological spaces, such as cgwh spaces.}

 Recall that diagram categories like $[\Drect,\mathsf{Top}]$ admit Hurewicz type model structures by \cite{barthel_construction_2013}. They also carry projective model structures constructed from $\mathsf{Top}_{\text{q}}$ by transfer. Hence, it is possible to generalize Propositions  \ref{prop_LocalizationsOfHurewiczMixedModelsSIMPLICIALCASE} and \ref{prop_LocalizationsOfHurewiczMixedModelsTOPOLOGICALCASE} to $[\Drect,\mathsf{Top}]$, producing semimodel categories that present ``exotic" homotopy categories, e.g. homotopy sheaves with values in topological spaces for which descent axioms hold up to strong homotopy equivalence, $\Upgamma$-spaces satisfying Segal's conditions up to strong homotopy equivalences or mixed versions of them. 
 \end{rem}

Regarding right Bousfield localizations, it is important to recall that there are no interesting ones for $\mathsf{Top}_{\text{q}}$ (and hence of $\mathsf{Top}_{\text{m}}$) unless we consider pointed spaces, as noted by Chach\'olski in his thesis\footnote{As pointed out by the anonymous referee, according to a remark in  \cite{farjoun_cellular_1996}, the famous joke ``\emph{considering colocalizations of unpointed spaces is pointless}" is due to G.Mislin.}. A modern way to justify this is that the $\infty$-category of $\infty$-groupoids, since it is generated under homotopy colimits by a point, only has as coreflexive subcategories the initial and the terminal ones. That does not happen for $\mathsf{Top}_{\text{h}}$, which can admit exotic right Bousfield localizations. In any case of interest, one has a kind of dual to Proposition \ref{prop_LocalizationsOfHurewiczMixedModelsSIMPLICIALCASE}.
\begin{prop}\label{prop_RBLTopologicalSpacesSIMPLICIALCASE}
Let $\Krect$ be a set of objects   in $\mathsf{Top}_{(*)}$. Then, the right Bousfield localization of $\mathsf{Top}_{(*),\textup{h}}$ (resp.\;$\mathsf{Top}_{*,\textup{m}}$) at $\Krect$-equivalences exists among coSpitzweck left semimodel categories. Its cofibrant objects are the $\Krect$-cellular spaces.
\end{prop} 
\begin{proof}
Apply the dual of Theorem \ref{thm_LBLfromLocalizator} to the obvious variations of the homotopy idempotent functor $\mathsf{CW}_{\Arect}$ constructed in \cite[Section 2.B]{farjoun_cellular_1996}.
\end{proof}



The dual to Proposition \ref{prop_LocalizationsOfHurewiczMixedModelsTOPOLOGICALCASE} would require a cellular version of the strong localization functor constructed in \cite{halperin_localization}. We leave this work to the interested reader.

\begin{rem}
A similar discussion applies if we replace $\mathsf{Top}$ by a category of chain complexes with either the triple of model structures (Hurewicz, mixed, projective) or either the triple (Hurewicz, mixed, injective) (see  \cite{barthel_six_2014} or  \cite[Chapter 18]{may_more_2012}). In this case, the different derived mapping objects are given by the internal hom and the canonical mapping space constructed via framings.
\end{rem}

\paragraph{Mixed (semi)model structures.} The existence of the mixed model structure, both for spaces and chain complexes, is a consequence of a method to combine model structures invented by Cole, see \cite[Section 17.3]{may_more_2012}. It  requires two nicely interacting model structures to produce the mixture, e.g.\ Hurewicz and Quillen model structures on spaces. In this paragraph, we give an alternative approach to find (a substantial part of) the mixed model structure equipped only with the Hurewicz part. This discussion should be compared with \cite[Section 19.1]{may_more_2012}.

The first way to obtain such structure is by means of Proposition \ref{prop_RBLTopologicalSpacesSIMPLICIALCASE}.
\begin{cor} The simplicial right Bousfield localization of $\mathsf{Top}_{\textup{h}}$ at the set  $\left\{\mathbb{S}^{n}\right\}_{n\geq 0}$
exists as a coSpitzweck left semimodel category and is Quillen equivalent to the mixed model structure $\mathsf{Top}_{\textup{m}}$ (and hence to the Quillen model structure $\mathsf{Top}_{\textup{q}}$).
\end{cor}
\begin{proof} The existence comes from Proposition \ref{prop_RBLTopologicalSpacesSIMPLICIALCASE} and the Quillen equivalence follows from the fact that the classes of equivalences coincide by \cite[Section 4]{arlin_detecting_2020}. 
\end{proof}

\begin{rem} Surprisingly, there does not exist a right Bousfield localization of $\mathsf{Top}_{\textup{h}}$ at a set of objects $\Krect$ presenting the homotopy theory of CW-complexes if we define the new equivalences to be those maps $\upf$ such that $\left[\,\Trect,\upf\;\right]$ is a bijection for any $\Trect\in \Krect$. The problem being that weak homotopy equivalences cannot be detected by any  family of ``representable" functors $\left[\,\Trect,\star\,\right]\colon \mathsf{Top}\to \mathsf{Set}$ with $\Trect$ varying in a fixed set because of \cite[Theorem 2.1]{arlin_detecting_2020}.
\end{rem}

A more direct way to obtain a mixed semimodel structure is by considering a particular CW-approximation. Let us consider for example the one associated to the usual  adjunction $\vert\star\vert\colon\mathsf{sSet}_{\textup{KQ}}\rightleftarrows\mathsf{Top}_{\textup{h}}\colon \mathsf{Sing}$, i.e. the BF-coreflector given by $\Xrect\mapsto \vert\mathsf{Sing}(\Xrect)\vert$.
\begin{prop} The right Bousfield localization of $\mathsf{Top}_{\textup{h}}$ associated to the BF-coreflector determined by $\vert\mathsf{Sing}(\Xrect)\vert\to \Xrect$ using  Theorem \ref{thm_LBLfromLocalizator} is Quillen equivalent to  $\mathsf{sSet}_{\textup{KQ}}$.
\end{prop}
\begin{proof}
	The result is obvious taking into account that the BF-coreflector is a CW-approximation, but a direct proof is also possible. That is, $\vert\star\vert\dashv \mathsf{Sing}$ induces a Quillen adjunction between $\mathsf{sSet}_{\textup{KQ}}$ and the right Bousfield localization and it is a Quillen equivalence by choice of BF-coreflector and since realization creates weak homotopy equivalences in $\mathsf{sSet}$.
\end{proof}

Of course, one can use different shapes (replacing the functor $\Delta\to \mathsf{Top}$ that determines $\vert\star\vert\dashv\mathsf{Sing}$) or different CW-approximation functors to produce this semimodel structure, but comparing them requires the effort of checking that the homotopy idempotent functors that we use are CW-approximations.

\begin{rem}
A similar phenomenom happens in homological algebra. The h-model structure on chain complexes (\cite{barthel_six_2014}) can be left or right localized by choising a functorial $\Krect$-flat (resp.\ $\Krect$-injective or $\Krect$-projective) resolution, provided it exists, to get a semimodel structure presenting the derived category. Of course, the existence of full model structures that present these derived categories has been addressed, but such results require much more effort than constructing resolutions.
\end{rem}

\paragraph{Relative homological algebra.} As observed in \cite{christensen_quillen_2002}, the homological algebra one can do in a bicomplete abelian category $\mathcal{A}$ with respect to a projective class $\mathcal{P}$ frequently fits into the model categorical framework. More concretely, a class of objects $\mathcal{P}$ in $\mathcal{A}$ determines a class of $\mathcal{P}$-equivalences and of  $\mathcal{P}$-fibrations in $\Ch(\mathcal{A})$:
\begin{itemize}
\item 	$f$ is an $\mathcal{P}$-equivalence iff  $\Hom(\Prect,f)$ is a quasi-isomorphism in $\Ch(\mathbb{Z})$ for all $\Prect\in \mathcal{P}$;
\item $f$ is a $\mathcal{P}$-fibration iff  $\Hom(\Prect,f)$ is an epimorphism in $\Ch(\mathbb{Z})$ for all $\Prect\in \mathcal{P}$.
\end{itemize}
One of their principal results, \cite[Theorem 2.2]{christensen_quillen_2002}, asserts that these classes yield a model structure $\Ch(\mathcal{A})_{\mathcal{P}}$ provided cofibrant replacements exist (and they provide  conditions for this to happen). For example, the h-model structure on chain complexes over a ring can be obtained this way (\cite[Remark 1.16]{barthel_six_2014}). 

Our interest on these model structures is that they can fail to be cofibrantly generated (\cite[Subsection 2.4]{christensen_quillen_2002}), but even in those cases one has:
\begin{thm}\label{thm:LBLOfRelativeHomologicalAlgebra} Let $\mathcal{A}$ be a presentable  abelian category and $\mathcal{P}$ a projective class on it such that the relative model structure $\Ch(\mathcal{A})_{\mathcal{P}}$ exists. Then, for any set of maps $\Srect$ in $\Ch(\mathcal{A})$, the left Bousfield localization of $\Ch(\mathcal{A})_{\mathcal{P}}$ at $\Srect$ exists as a coSpitzweck right semimodel category.  
\end{thm}
\begin{proof} Immediate from Theorem \ref{thm_LBLatSequivalencesWith1StepLocalizator} because: (i) $\Ch(\mathcal{A})_{\mathcal{P}}$ is left proper by \cite[Proposition 2.8]{christensen_quillen_2002}; (ii) $\Ch(\mathcal{A})$ is presentable and so any set of maps admits Quillen's soa; (iii) any object in $\Ch(\mathcal{A})_{\mathcal{P}}$ is fibrant by definition and thus we can take $\Jrect=\emptyset$ in the definition of $\overline{\mathsf{Horn}}(\Srect)$.
\end{proof}

\begin{rem} When $\Ch(\mathcal{A})_{\mathcal{P}}$ is cofibrantly generated, its left Bousfield localization at any set of maps exists as a full model structure since $\Ch(\mathcal{A})_{\mathcal{P}}$ is then combinatorial and left proper. Hence, when $\Ch(\mathcal{A})_{\mathcal{P}}$ is cofibrantly generated, we are not getting anything new. However, this is not the case when $\Ch(\mathcal{A})_{\mathcal{P}}$ is not cofibrantly generated.
\end{rem}

\paragraph{Truncations and t-structures.} 
 Recall that a t-structure on a triangulated category $\EuScript{T}$ (\cite[Section IV.4]{gelfand_methods_2003}) yields a way to find an abelian category inside $\EuScript{T}$. The idea behind t-structures is to collect the formal properties that (co)connective complexes have within the unbounded derived category of an abelian category, because its intersection, $\EuScript{T}^{\leq 0}\cap\EuScript{T}^{\geq 0}$, recovers the initial abelian category. However, the notion of a t-structure has important applications apart from this motivational idea, for instance: they are fundamental to define perverse sheaves; they are related to Posnikov towers and spectral sequences; and surprisingly, from the author point of view, the existence of a certain t-structure implies Grothendieck's standard conjectures. In this brief application, we discuss how lifting a t-structure to the model category level yields several Bousfield localizations, which can be interesting for example due to their relation to Postnikov towers.

Suppose given a stable model category $\D$ whose homotopy category carries a t-structure $(\Ho\D^{\leq 0},\Ho\D^{\geq 0})$ (see \cite[Definition IV.4.2]{gelfand_methods_2003}), such as cochain complexes with the mixed model structure and usual truncation functors. Then, one could ask for a functorial lifting of the distinguished triangle
$$
\Xrect^{\leq 0} \to \Xrect\to \Xrect^{\geq 1}\xrightarrow{+1} 
$$
into the model categorical level. Provided such lift is given by a BF-coreflector $\uptau^{\leq 0}\Xrect \to \Xrect$ and a BF-reflector $\Xrect\to \uptau^{\geq 0}\Xrect$ on $\D$, i.e.\ essentially the truncations can be made in a point-set level, one obtains directly from Theorem \ref{thm_LBLfromLocalizator}:
\begin{prop}\label{prop_TruncatedModelsForTstructure} The model structure $\D$ admits a left (resp.\ right) Bousfield localization as a coSpitzweck right (resp.\ left) semimodel category which presents the category $\Ho\D^{\leq m}$ (resp.\ $\Ho\D^{\geq m}$) associated to the t-structure on $\Ho\D$. 
\end{prop}

It is even possible to study the homotopy theory of (co)complete objects of $\D$ with respect to the t-structure.
\begin{defn} An object $\Xrect\in \D$ is $\uptau $-\emph{complete} (resp.\ $\uptau$-\emph{cocomplete}) if it can be reconstructed from its truncations $\uptau^{\geq m}\Xrect$ (resp.\ $\uptau^{\leq m}\Xrect$), i.e. if the map 
	$
	\Xrect\rightarrow \holim_{m}\uptau^{\geq m}\Xrect
	$  
(resp.\ 
$
\hocolim_{m}\uptau^{\leq m}\Xrect\to \Xrect
$)	
is an equivalence. 
\end{defn}

\begin{prop}\label{prop_CompleteModelsForTstructure} The model structure $\D$ admits a left (resp.\ right) Bousfield localization as a coSpitzweck right (resp.\ left) semimodel category which presents the homotopy theory of $\uptau$-complete (resp.\ $\uptau$-cocomplete) objects associated to the t-structure on $\Ho\D$. 
\end{prop}
\begin{proof}
The natural transformation $\id\Rightarrow \holim_m\uptau^{\geq m}$  yields a BF-reflector on $\D$ whose associated left Bousfield localization through Theorem \ref{thm_LBLfromLocalizator} presents the homotopy theory of complete objects in $\D$. Analogous for $\hocolim_{m}\uptau^{\leq m}\Rightarrow \id$.
\end{proof}

\paragraph{Aqfts  and locally constant factorization algebras.}
Let us elaborate on  two applications of Theorem \ref{thm_LBLfromLocalizator} related to algebraic quantum field theories and factorization algebras. 

It is well-known how model categories have been applied to codify descent conditions, that is, to construct homotopy theories of sheaves, stacks and higher generalizations. Underneath this codification lies the fact that the involved descent conditions can be encoded as locality with respect to a set of maps, i.e. fibrant objects in a left Bousfield localization. On the contrary, the less common but not less important notion of codescent is way more elusive in terms of incorporating it into a model categorical framework. Let us briefly recall two instances of such codescent conditions.

In the setting of algebraic quantum field theories (diagrams of algebras satisfying some axioms), there is a codescent condition which asks for reconstruction of global observables in a physical system in terms of observables measured in special regions, e.g. causal diamonds. See \cite{carmona_algebraic_2021} and the references therein. The dual of Theorem \ref{thm_LBLfromLocalizator} is a key ingredient in the establishment and the recognition of cofibrant objects in the following:
\begin{thm}{\cite[Theorem B]{carmona_algebraic_2021}} There is a model structure on the category of AQFTs over an orthogonal category $\Crect^{\perp}$, denoted  $\mathcal{QFT}_{\Crect_{\diamond}}(\Crect^{\perp})$, which present the homotopy theory of AQFTs satisfying an appropiate $\Crect_{\diamond}$-local-to-global principle. 
\end{thm} 

In the theory of factorization algebras (precosheaves with certain multiplicative structure), there is also an important codescent condition: factorization algebras are cosheaves for the Weiss topology. See \cite{carmona_model_2021} and the references therein for definitions and background on this topic. Again, the dual of Theorem \ref{thm_LBLfromLocalizator} is fundamental to prove the existence of certain right Bousfield localizations in \cite{carmona_model_2021}. This result combined with Theorem \ref{thm_BousfieldLocalizationCellularCase} are essential to prove the following instance of the main theorem in that paper.
\begin{thm}{\cite[Theorem A]{carmona_model_2021}}
	There is a (Spitzweck) left semimodel structure on the category of prefactorization algebras over a manifold with values in $\mathsf{Top}_{\textup{q}}$ presenting the homotopy theory of locally constant factorization algebras over that manifold.
\end{thm} 

\paragraph{Stabilization of model categories.} Hovey developed a pair of stabilization machines in \cite{hovey_spectra_2001} whose applicability can be extended to cellular (resp. combinatorial) non left-proper settings by Theorem \ref{thm_BousfieldLocalizationCellularCase} (resp. \cite[Theorem A]{white_left_2020}) as observed in \cite{white_left_2020}. Additionally, we present a different stabilization machine which applies to a completely different class of model categories.

Imitating \cite{hovey_spectra_2001}, we consider a stabilization context to be given by a model category $\M$ equipped with a Quillen adjunction of endofunctors $\Trect\colon\M\rightleftarrows\M:\Upomega$. From this data, one can form a category of $\Trect$-prespectra $\mathsf{Spt}_{\Trect}(\M)$ \cite[Definition 1.1]{hovey_spectra_2001} and equip it with a level model structure characterized by: a map $\mathcal{X}\to\mathcal{Y}$ is a level-equivalence (resp. level-fibration) if $\mathcal{X}(n)\to\mathcal{Y}(n)$ is an equivalence (resp. fibration) $\forall\upn\geq 0$. Note that $\M$ is not required to be cofibrantly generated, since this level-model structure can be constructed by hand. The $\Trect$-stabilization of $\M$ will be a left Bousfield localization of this level-model structure on $\mathsf{Spt}_{\Trect}(\M)$ designed to present the homotopy theory of $\Upomega$-spectra. Recall that a $\Upomega$-spectrum is a $\Trect$-prespectrum $\mathcal{X}\in\mathsf{Spt}_{\Trect}(\M)$ which satisfies that all its maps $\mathcal{X}(n)\to \Upomega\mathcal{X}(n+1)$ are equivalences.

%
%

The following technique consists on a generalization of \cite[Proposition 2.1.5]{schwede_spectra_1997} which drops some hypotheses in that result at the expense of producing just a semimodel category. The idea is to find a BF-reflector $\upeta\colon\mathcal{X}\to\Qstab\mathcal{X}$ whose $\Qstab$-local objects are essentially the $\Upomega$-spectra.

Let us give a preliminary construction which is analogous to that in \cite[Subsection 2.1]{schwede_spectra_1997}: given $\mathcal{X}\in\mathsf{Spt}_{\Trect}(\M)$, we will define a sequential family of $\Trect$-prespectra $(\mathcal{Z}_{r})_{r}$ connected by maps $\mathcal{Z}_{r}\to\mathcal{Z}_{r+1}$ which are levelwise cofibrations. Set $\mathcal{Z}_0=\mathcal{X}$ and construct $\mathcal{Z}_{1}$ by: for each $n\geq 0$, consider the commutative square
$$
\begin{tikzcd}[ampersand replacement=\&]
\mathcal{X}(n)\ar[r, rightarrowtail]\ar[d] \& \mathcal{Z}_{1}(n)\ar[d]\ar[ld, twoheadrightarrow, "\simeq" description]\\
\Upomega\mathcal{X}(n+1) \ar[r] \& \Upomega\mathcal{Z}_{1}(n+1)
\end{tikzcd}.
$$
Originally, we only have the left vertical map, but an application of the functorial factorization $(\Cof,\AFib)$ in $\M$ yields the upper horizontal map and the diagonal. The lower horizontal arrow comes from applying $\Upomega$ to the $(n+1)$ version of the upper one.
Repeat this process inductively to get the sequence of $\Trect$-prespectra 
$
\mathcal{X}=\mathcal{Z}_0\to\mathcal{Z}_1\to\cdots\to\mathcal{Z}_{r}\to \cdots
$
and set $\overline{\Qstab}\mathcal{X}=\mathsf{colim}_{r}\,\mathcal{Z}_{r}.$ 
Note that the structural maps $\overline{\Qstab}\mathcal{X}$ can be defined if we assume that $\Upomega$ preserves sequential colimits of cofibrations and that by construction $\mathcal{Z}_{r}(n)\simeq \Upomega^{r}\mathcal{X}(n+r)$. 

Define $(\Qstab,\upeta)$ as the composition of the previous construction and a level-fibrant replacement in $\mathsf{Spt}_{\Trect}(\M)$ (which is functorial since $\M$ is supposed to admit functorial factorizations), i.e.
$$
\begin{tikzcd}[ampersand replacement=\&]
\upeta\colon\mathcal{X}\ar[rr,"\text{fibrant}","\text{replacement}"'] \& \& \mathsf{R}\mathcal{X} \ar[rr, "\text{colim}", "\text{component}"'] \&\& \overline{\Qstab}\mathsf{R}\mathcal{X}=\Qstab\mathcal{X}.
\end{tikzcd}
$$

\begin{lem}\label{lem_StabilizationReflector} Let $\M$ be a model category with functorial factorizations in which sequential colimits of cofibrations are homotopical. Assume that $\M$ is equipped with a Quillen adjunction $\Trect\colon\M\rightleftarrows\M:\Upomega$ such that $\Upomega$ sends sequential colimits of cofibrations to homotopy colimits. Then, the pair $(\Qstab,\upeta)$ is a BF-reflector on the level-model structure on $\mathsf{Spt}_{\Trect}(\M)$.
\end{lem}
\begin{proof}  We should prove the following points:	
	\begin{itemize}
		\item $\Qstab$ is homotopical with respect to level-equivalences: a level-equivalence $\mathcal{X}\to\mathcal{Y}$ induces equivalences $\Upomega^{r}\mathsf{R}\mathcal{X}(n+r)\simeq \Upomega^{r}\mathsf{R}\mathcal{Y}(n+r)$ for any $(r,n)$ since $\Upomega$ is right Quillen. An inspection of the construction of the sequence $(\mathcal{Z}_{r})_r$ associated to a $\Trect$-prespectrum shows that we obtain a level-equivalence $\Qstab\mathcal{X}\to\Qstab\mathcal{Y}$ provided sequential colimits of cofibrations are homotopical in $\M$.
		\item $\upeta\Qstab$ is a level-equivalence: it suffices to prove that $\overline{\Qstab}\mathcal{X}$ is a $\Upomega$-spectrum if $\mathcal{X}$ is level-fibrant, because if this is the case, the sequence $\widetilde{\mathcal{Z}}_0\to\widetilde{\mathcal{Z}}_{1}\to \cdots$ that defines $\upeta\Qstab$ is comprised of acyclic cofibrations. To see that $\overline{\Qstab}\mathcal{X}$ is a $\Upomega$-spectrum simply apply that the structural maps $\overline{\Qstab}\mathcal{X}(n)\to\Upomega\overline{\Qstab}\mathcal{X}(n+1)$ can be obtained as the sequential colimit of the acyclic fibrations  $\mathcal{Z}_{r}(n)\twoheadrightarrow\Upomega\mathcal{Z}_{r-1}(n+1)$ in the construction of $(\mathcal{Z}_r)_r$. The hypotheses on $\M$ and $\Upomega$ yield that the structural maps $\overline{\Qstab}\mathcal{X}(n)\to\Upomega\overline{\Qstab}\mathcal{X}(n+1)$ are equivalences. 
		
		\item $\Qstab\upeta$ is a level-equivalence: an easy diagram chasing using the previous item and 2-out of-3 for equivalences in $\M$.
	\end{itemize}
\end{proof}

Combining Lemma \ref{lem_StabilizationReflector} and Theorem \ref{thm_LBLfromLocalizator}, we obtain:
\begin{thm}\label{thm_StabilizationOfModelCategories} Let $\M$ be a model category with functorial factorizations in which sequential colimits of cofibrations are homotopical. Assume that $\M$ is equipped with a Quillen adjunction $\Trect\colon\M\rightleftarrows\M:\Upomega$ such that $\Upomega$ sends sequential colimits of cofibrations to homotopy colimits. Then, the level-model structure on $\mathsf{Spt}_{\Trect}(\M)$ admits a left Bousfield localization among (coSpitzweck) right semimodel categories, denoted $\Spt_{\Trect}(\M)$ and called stable semimodel structure, which presents the homotopy theory of $\Upomega$-spectra over $\M$. 	
\end{thm}
\begin{proof} We are reduced to check that $\Upomega$-spectra are essentially $\Qstab$-local objects.
	
Let $\mathcal{X}$ be a level-fibrant $\Upomega$-spectrum. Then, the associated sequence $(\mathcal{Z}_r)_r$ is comprised of  levelwise acyclic cofibrations $\mathcal{Z}_r\to\mathcal{Z}_{r+1}$. Therefore, $\mathcal{X}\to\Qstab\mathcal{X}$ is a level-equivalence.
	
Viceversa, it suffices to consider $\mathcal{X}$ to be a level-fibrant $\Trect$-prespectrum such that $\mathcal{X}\to \overline{\Qstab}\mathcal{X}$ is a level-equivalence. If we show that in this case $\overline{\Qstab}\mathcal{X}$ is a $\Upomega$-spectrum we will conclude the result. This fact was proven in Lemma \ref{lem_StabilizationReflector}, when showing that $\upeta\Qstab$ is a level-equivalence.
\end{proof}

Comparing the former proofs with those of \cite{schwede_spectra_1997}, one obtains also a semimodel version of Theorem \ref{thm_StabilizationOfModelCategories} for non right proper model categories that admit \emph{the small object argument} in the sense of loc.cit.\footnote{These model structures may not be combinatorial nor cellular.}. Also, note that left properness in the cited reference is not necessary for the conclusion that we are pursuing.

\begin{cor} The stabilization of $\mathsf{Top}_{\textup{h}}$ in the usual sense is presented by the stable semimodel structure  $\Spt_{\Upsigma}(\mathsf{Top}_{*,\textup{h}})$ obtained in Theorem \ref{thm_StabilizationOfModelCategories}. The analogous result holds for chain complexes in non-negative degrees with the Hurewicz model structure.
\end{cor}
\begin{proof}
	The condition on sequential colimits is automatic for chain complexes, since all objects are $\textup {h}$-cofibrant \cite[Theorem 18.3.1]{may_more_2012}. For pointed topological spaces, this assumption follows from left properness \cite[Corollary 17.1.2]{may_more_2012}. Hence, one is reduced to check the condition on $\Upomega$ to apply the cited result in both model structures.	
	
    For chain complexes, $\Upomega$ is a degree shifting plus a truncation $\uptau_{\geq 0}$. Hence, $\Upomega$ preserves sequential colimits and h-cofibrations, since h-cofibrations are degreewise split-monomorphisms by \cite[Proposition 1.14]{barthel_six_2014}, and thus $\Upomega$ preserves the required homotopy colimits. 
	
	For spaces, $\Upomega$ is the pointed mapping space functor  $\Map_{*}(\mathbb{S}^1,\star)\colon \mathsf{Top}_{*,\textup{h}}\to \mathsf{Top}_{*,\textup{h}}$. The required property holds since $\mathbb{S}^{1}$ is compact and strong Hurewicz cofibrations are closed inclusions by \cite[Lemma 5.9]{strickland_category_nodate}.
\end{proof}

\begin{rem} The class of model categories covered by Theorem \ref{thm_StabilizationOfModelCategories} cannot be stabilized by any other existing method to the best of the author knowledge. One essential obstruction is that model categories like $\mathsf{Top}_{\textup{h}}$ do not admit a set of objects that detect equivalences.
\end{rem}

	
Also, it will be interesting to combine Theorem \ref{thm_StabilizationOfModelCategories} with Propositions \ref{prop_TruncatedModelsForTstructure}	and \ref{prop_CompleteModelsForTstructure}.

\paragraph{Topological examples in Batanin-White's list.} In \cite{white_left_2020}, Batanin-White list a great deal of novel possible applications of \cite[Theorem A]{white_left_2020}. Some of those applications do not naturally fit into the combinatorial framework whereas they can be proved directly by Theorem \ref{thm_BousfieldLocalizationCellularCase} since the involved model categories are cellular. Anyway, Batanin and White deserve the credit of having proposed combinatorial alternatives for those examples. 

For instance, Examples 5.8 and 5.9 in that paper, which are respectively studied in \cite{goerss_moduli_2005} and \cite{harper_topological_2019}, are actual precursors of Theorem \ref{thm_BousfieldLocalizationCellularCase}. Examples 5.11 and 5.12 are originally proposed for model categories derived from topological spaces, and so they cannot be combinatorial, but they can be cellular. Also, the use of topological methods could also help in other applications like Example 5.17.

All these applications have in common that the model structures that one is willing to localize are of topological nature. The interested reader is strongly recommended to look at \cite{white_left_2020}. Here, we propose a short list of types of model structures that can be localized with Theorem \ref{thm_BousfieldLocalizationCellularCase}\footnote{Recall from Remark \ref{rem_EffectiveMonoForIntersections} that the effective monomorphism condition on cellular categories can be relaxed.}: (1) operadic algebras over cellular model categories like topological spaces with the Quillen model structure, orthogonal spectra, $\mathbb{S}$-modules...; (2) right Bousfield localizations of cellular model structures with all objects fibrant, e.g. projective model structures on diagrams of spaces; (3) resolution type model structures as in \cite{goerss_moduli_2005}.

\begin{paragraph}{Acknowledgments}
	The author would like to thank his advisors, Ramon Flores and Fernando Muro, for their support and time. He also wants to thank Carles Casacuberta, Emmanuel Dror-Farjoun, Carlos Maestro and David White for their disposition, kind attention and interest. Finally, he thanks the anonymous referee for suggesting Christensen-Hovey's model categories for relative homological algebra as a source of examples.
\end{paragraph}

\bibliography{Bibliography}

	$\,$\\
\textsc{Universidad de Sevilla, Facultad de Matem\'aticas, Departamento de \'Algebra-Imus, Avda. Reina Mercedes s/n, 41012 Sevilla, Spain}

\textit{Email address:} \texttt{vcarmona1@us.es}

\textit{url:} \texttt{http://personal.us.es/vcarmona1}

\end{document}